\documentclass[a4paper,11pt]{amsart}
\usepackage{amsmath,amsfonts,amssymb}
\usepackage{enumerate}
\usepackage{lineno}
\usepackage{xspace}
\usepackage{xcolor}
\usepackage{graphicx}
\usepackage{subcaption}

\usepackage{epstopdf}
\usepackage{esvect}
\usepackage{bbm}
\usepackage[normalem]{ulem}
\usepackage[backref=page]{hyperref}

\newtheorem{assumption}{Assumption}

\newtheorem{lemma}{Lemma}

\newtheorem{definition*}{Definition}
\newtheorem{theorem}{Theorem}
\newtheorem{remark}{Remark}
\newtheorem{acknowledgements}{Acknowledgements}
\renewcommand{\vec}[1]{\mathbf{#1}}

\newcommand\norm[1]{\left\lVert#1\right\rVert}

\newcommand\Lpnorm[1]{\left\lVert#1\right\rVert _ {L^p(\Omega,\mathbb{R}^d)} }

\newcommand{\expect}[1]{\mathbb{E}{\left[#1\right]}}
\newcommand\Lnorm[2]{\left\lVert#1\right\rVert _ {L^{#2}(\Omega,\mathbb{R}^d)} }
\newcommand\LnormR[2]{\left\lVert#1\right\rVert _ {L^{#2}(\Omega,\mathbb{R})} }

\newcommand\eval[1]{\mathbb{E}\left[ #1  \right]    }

\newcommand{\der}[2]{\frac{\partial #1} {\partial #2} }

\newcommand{\innerproduct}[2]{{ \big \langle} #1 , #2 { \big \rangle}}

 \newcommand{\Yht}{\bar{Y}_{t_n} }

 \newcommand{\Yhk}{\bar{Y}_{t_k} }

 \newcommand{\Xt}{X_t}
 \newcommand{\Yt}{\bar{Y}_t }

 \newcommand{\Ys}{\bar{Y}_s }

 \newcommand{\Yr}{\bar{Y}_r }
 
\newcommand{\Lh}{\mathcal{\hat{L}}}
\newcommand{\Lhi}{\mathcal{\hat{L}}^i}

\newcommand{\LL}{\mathcal{L}}

\newcommand{\mut}{\mu^{tm}}
\newcommand{\sigmat}{\sigma_i ^{tm}}

\newcommand{\mutp}[1]{\mu^{tm} ( #1 )  }
\newcommand{\sigmatp}[1]{\sigma_i ^{tm}   ( #1 )  }

\newcommand{\Ft} {F^{tm} _{\Dt}}
\newcommand{\FG}{\mathcal{N}}
\newcommand{\NZ}[1] {\mathbb{N}_{#1}}
\newcommand{\NP}[1] {\mathbb{N}_{#1}^+}
\newcommand{\SG} [1]{\mathbf{\Phi} _{#1} }
  
\newcommand{\Dt}{\Delta t}
\newcommand{\DW}[1]{\Delta W^i_#1}
\newcommand{\vDW}[1]{\Delta \vec{W}_#1}

\newcommand{\sB}{\sum_{i=1}^mB_i}
\newcommand{\sg}{\sum_{i=1}^mg_i}
\newcommand{\sumk}{\sum_{k=0}^{n-1} }
\newcommand{\sumi}{\sum_{i=1}^{m} }
\newcommand{\intk}{\int_{t_k}^{t_{k+1}} }
\newcommand{\inttn}{\int_{t_n}^t }

\newcommand{\Ito}{It\^o\xspace}
\newcommand{\Lref}{{N_{R}}\xspace}

\renewcommand{\d}{{\mathrm{d}}}
\newcommand{\dt}{{\d}{t}}

\newcommand{\dW}{{{\d}W}}
\newcommand{\ds}{{\d}{s}}
\newcommand{\Y}{{\mathbb{Y}}}

\begin{document}

\title[Weak Convergence Of Tamed Exponential Integrators]{Weak Convergence of Tamed Exponential Integrators for Stochastic Differential Equations}

\author{Utku Erdo\u{g}an}
\address{Department of Mathematics, Eski\c{s}ehir Technical University, YunusEmre Kamp\"{u}s\"{u}, Eski\c{s}ehir, 26470, T\"{u}rkiye}
\email{utkuerdogan@eskisehir.edu.tr}

\author{Gabriel J. Lord}
\address{Department of Mathematics, IMAPP, Radboud University, Nijmegen, 6500 GL, Netherlands}
\email{gabriel.lord@ru.nl}

\maketitle

\begin{abstract}
 We prove weak convergence of  order one for a class of exponential based integrators for SDEs with non-globally Lipschitz drift. Our analysis covers tamed versions of Geometric Brownian Motion (GBM) based methods as well as the  standard exponential schemes. The numerical performance of both the GBM and exponential tamed methods through four different multi-level Monte Carlo  techniques are compared. We observe that for linear noise the standard exponential tamed method requires severe restrictions on the step size unlike the GBM tamed method.
\end{abstract}

\section{Introduction}
\label{intro}
We consider weak convergence analysis of tamed schemes for semi-linear stochastic differential equations (SDEs) of the form
\begin{equation} \label{eq:EqAFB}
  \d\Xt= \left( A \Xt+F(\Xt) \right) \dt + \sum _{i=1}
  ^m  \left(  B_i \Xt+g_i(\Xt)\right)  \dW_t^i,
  \quad X_0=x \in  \mathbb{R} ^d, 
\end{equation}
with $m$, $d\in\mathbb{N}$. 
Here $W_t^i$ are iid Brownian motions  on the probability space  $(\Omega , \mathcal{F},\mathbb{P})$   with filtration $({\mathcal{F} _ {t}  }) _{t \in 
  [0,T]}$. The nonlinearity $F : \mathbb{R} ^d \rightarrow \mathbb{R} ^d$ only satisfies a one-sided Lipschitz condition, where as the $g_i : \mathbb{R} ^d \rightarrow \mathbb{R} ^d$ are globally Lipschitz.
  We assume throughout that the matrices $A$, $B_i \in
\mathbb{R} ^{d \times d }$  satisfy the following zero commutator conditions 
\begin{equation} \label{eq:commute}
AB_i-B_iA=0, \quad B_jB_i-B_iB_j=0
\qquad \text{for} \quad  i,j=1, \hdots, m.
\end{equation}
Condition \eqref{eq:commute} is used to exploit the exact solution of Geometric Brownian Motion (GBM) below.
By introducing the 
notation $W_t:=[W^1_t,\ldots,W^m_t]^T$,
\begin{equation}
\mu (y):= A y+F(y), \quad  \sigma_i(y) :=B_i y+ g_i(y),
\label{eq:musigmadefn}
\end{equation}
and the matrix $\sigma(x):=[\sigma_i(x)]$, 
(i.e. columns formed by  $\sigma_i$) we can re-write 
\eqref{eq:EqAFB} as 
\begin{equation}\label{eq:mu_sigma_SDE}
\d\Xt= \mu(\Xt) \dt + \sigma(\Xt) \dW_t,  \quad X_0=x \in  \mathbb{R} ^d.
\end{equation}
The schemes we examine are in the class of exponential integrators. These methods have proved to be effective schemes for many SDEs and stochastic partial differential equations (SPDEs).
Originally only linear (or linearized) drift terms were exploited, see for example \cite{lord2004,biscay1996,Mora,jimenez1999simulation,JimenezCarbonell}, however recently several  exponential integrators that also exploit the linear terms in diffusion emerged \cite{expmil,utkuLord,debrabant2021rk,yang2021class}.
We are particularly interested in dealing with a one-sided Lipschitz drift term $F$ with superlinear growth. Hutzenthaler \textit{et al.} \cite{hutzenthaler2011} showed both strong and weak divergence of Euler's method for the general SDE  \eqref{eq:EqAFB} with superlinearly growing coefficients $F$ and/or $g_i$. Following \cite{hutzenthaler2011} many explicit variants of Euler--Maruyama schemes that guarantee the strong convergence to the exact solution of SDE were derived, see for example \cite{hutzenthaler2012,TruncatedMao2015,Beyn2016,TruncatedMao2018,Izgi2018,TretyakovZhang}. The  most well known approach to deal with superlinearly growing coefficients is to employ ``taming" to prevent the unbounded growth of numerical solutions. Although there has been much consideration of strong convergence of tamed schemes, see for example \cite{hutzenthaler2012,sabanis2013note,sabanis2016}  there has been little consideration of weak convergence for SDEs \cite{brehier2020weakergodic,bossy2021weak,wang2021weak,wang2021weakIMA}.
In \cite{brehier2020weakergodic}, Br{\'e}hier investigated the weak error of the explicit tamed Euler scheme for SDE's with one--sided Lipschitz continuous drift and additive noise  to approximate averages with respect to the invariant distribution of the continuous time process.  Bossy \textit{et al.}, \cite{bossy2021weak}, proposed  and proved weak convergence of a  semi-explicit exponential Euler scheme for a one-dimensional SDE with non-globally Lipschitz drift and diffusion behaving as $x^{\alpha}$, with $\alpha>1$. Due to the weak condition on the diffusion coefficient, their study covers regularity results for the solution of the Kolmogorov PDE commonly used in weak error  analysis. In  \cite{wang2021weakIMA}, Wang \textit{et al.} formulated a general weak convergence theorem for one-step numerical approximations of SDEs with non-globally  drift coefficients. They applied this to prove weak convergence of rate one for the tamed and backward Euler--Maruyama methods. We would like to point out that their analysis is not directly applicable to our GBM based schemes as it is not a classical one-step method but rather the composition of the GBM flow and a one-step flow.
In the context of SPDEs,
  Cai \textit{et al.}  \cite{cai2021weak} constructed and analysed a  weak convergence of a numerical scheme based on a spectral Galerkin method in space and a tamed version of the exponential Euler method. Below we impose conditions that are similar to those for the SPDE in \cite{cai2021weak}.
We prove weak convergence for a class of exponential integrators where a form of taming is used for the one-sided Lipschitz drift term. The GBM methods exploit the exact solution of geometric Brownian motion, see \cite{utkuLord} and \cite{tamedGBM} where strong convergence of related methods were considered.
Further, by taking $A=B_i=0$ (or incorporating these terms into the nonlinearities)
we simultaneously prove weak convergence for the standard exponential tamed scheme such as in 
\cite{CHEN2020135}.
Our proof is based on the Kolmogorov equation and one of the main difficulties is to take into account the stochasticity in the solution operator.

In our numerical experiments we compare different approaches to estimate the weak errors all using multi-level Monte Carlo (MLMC) techniques as reviewed in \cite{giles_2015}.
For a linear diffusion term we observe that the exponential tamed method does not perform well for larger time step sizes and hence a time step size restriction is required for MLMC techniques (for example to estimate the weak errors). This is of particular interest as tamed methods were originally introduced and strong convergence was examined precisely to control nonlinearities in the context of MLMC type simulations, see  \cite{hutzenthaler2012}. The GBM based method does not suffer in this way in our experiments for linear noise. For nonlinear diffusion both tamed based methods require a step size restriction for convergence on the MLMC techniques.

The paper is organized as follows: in Section \ref{sec:setting} we state our assumptions on the drift and diffusion, present the new numerical method and state our main results. In Section \ref{sec:NumericalResults} we present numerical simulations illustrating the rate of convergence using the MLMC simulations and compare the different approaches. The proofs of the main results are then given in detail in Sections \ref{sec:BoundedMoments} and \ref{sec:WeakConvergence}. 

\section{Setting and Main Results}
\label{sec:setting}
Throughout the paper we let $\innerproduct{\cdot}{\cdot}$ denote the standard inner product in $\mathbb{R}^d$ (so 
$\innerproduct{y}{z}=y ^\intercal z$ for $y,z \in \mathbb{R}^d$  )
and $\norm{ \cdot }$ represent both the Euclidean norm for vectors as well as the induced matrix norm. A vector $\beta=(\beta_1,\beta_2,\hdots,\beta_d)$ is a multiindex of order $\vert\beta\vert=\sum_{i=1}^d \beta_i$ with nonnegative integers components. The partial derivative operator corresponding to the multiindex $\beta$  is defined as
$$
D^\beta h(x)=\frac{\partial ^{\vert \beta \vert} h(x)}{\partial _{x_1}^{\beta_1} \partial _{x_2}^{\beta_2}\hdots \partial _{x_d}^{\beta_d} } 
$$
where $h\in C^{\vert \beta\vert}(\mathbb{R}^d;\mathbb{R}^l)$. For a nonnegative integer $j$,  
we let $\vec{D}^jh(x)$  represent the $j$th order {derivative operator} applied to a function $h\in C^j(\mathbb{R}^d;\mathbb{R}^l)$. When $j=1$ we simply write the Jacobian as $\vec{D}h$. 

 Additionally, $C^{k} _{b} (\mathbb{R}^d;\mathbb{R})$  denotes  of the set of k-times differentiable functions, which are uniformly continuous and bounded  together with  their derivatives up to k-th order.

We define the sets 
$$\NZ{n}:=\{0,1,2,\ldots,n\} \qquad \text{and} \qquad \NP{n}:=\{1,2,\ldots,n\}.$$ 
Before introducing our class of numerical methods we present three  results from \cite{cerrai} on the existence and uniquenes, bounded moments and mean-square differentiability of the exact solution to \eqref{eq:mu_sigma_SDE}. 
%

\subsection{Preliminary Results  for the  SDE}
For an SDE such as \eqref{eq:mu_sigma_SDE} 
with  globally Lipschitz drift and diffusion coefficients many classical textbooks on stochastic analysis consider the Kolmogorov PDE. However, for non-globally Lipschitz drift coefficients there are far fewer results. One key work is Cerrai \cite{cerrai} for the properties of the exact solution to a SDE with one-sided drift coefficient.

\begin{assumption}\label{ass:1}  Let $H\geq 0$ be given. Let the functions $F$ and $g_i \in C^h(\mathbb{R}^d;\mathbb{R}^d)$ for  some $h\geq H$ where $i=1,2, \hdots, m$. Define the matrix $g$ by the columns of $g_i$ so that $g=[g_i]$.
Additionally, assume that 

\begin{enumerate}[(i)]
\item   there exists $r \geq 0$ such that for any  $j=0,1,\hdots,h$

$\sup_{y \in \mathbb{R}^d} \norm{D^\alpha F(y)} (1+\norm{y}^{2r+1-j})^{-1} < \infty, \qquad  \vert \alpha \vert=j$;
\item there exists $\rho \leq r$ such that for any $j=0,1,\hdots,h$

$\sup_{y \in \mathbb{R}^d} \norm{D^\alpha g_i(y)} (1+\norm{y}^{\rho-j})^{-1} < \infty,\qquad  \vert \alpha \vert=j $;
\item for all $p>0$ there exist  $K=K(p) \in \mathbb{R}$ such that 

$ \innerproduct{y}{\vec{D} F(z) y}+ p \norm{\vec{D}g (z)y }^2 \leq K \norm{y}^2, \quad \forall \ y,z \in \mathbb{R}^d.$
\end{enumerate} 
\end{assumption}

\begin{assumption}\label{ass:2} 
There exist constants $a>0$ and $r,\gamma,c \geq 0$ such that for any $y,z \in \mathbb{R}^d$
$$\innerproduct{Az}{z} + \innerproduct{ F(y+z)-F(y)}{z} \leq -a \norm{z}^{2r+2} + c (\norm{y} ^{\gamma} +1).$$
\end{assumption}
In particular, under Assumption \ref{ass:1}, Cerrai \cite{cerrai}  proves the existence and uniqueness of a solution to the SDE \eqref{eq:EqAFB}.  
\begin{theorem}[\cite{cerrai},Theorem 1.3.5]
Suppose that Assumption \ref{ass:1} holds with $H=3$. Then there exists a unique solution $X_t$ for $t \in [0,T]$ to the SDE \eqref{eq:EqAFB} along with the following moment bound for $p \geq 1$ and constant $C=C(p,T)>0$  
\begin{equation} \label{eq:exact_boundedness}
   \eval{  \norm{X_t}^p} <C ( 1+   \norm{x}^p  ).
\end{equation}
\end{theorem}

To get order one weak convergence we need to assume bounded moments of  derivatives of the exact solution to the SDE \eqref{eq:EqAFB} with respect to the initial condition.  By the notation  $X_t ^x$, we emphasize  that  the initial condition is  $X_0=x$. We denote the  derivative of  the exact  solution with respect to the initial condition by $\vec{D}_{x}X_t ^x$. The following regularity result is  given in \cite{cerrai} (see also \cite{wang2021weakIMA,wang2021weak}).

\begin{theorem} [\cite{cerrai}, Theorem 1.3.6] \label{teo:bounded_derivatives}
Let Assumption \ref{ass:1} hold with $H=3$, Assumption \ref{ass:2} hold and $X_t^x$ be the solution to \eqref{eq:EqAFB}. Then $X_t^x$ is $h$ times mean-square differentiable and 
for $i=1,\hdots, h$, $ p\geq 1$ and $t \in [0,T]$
\begin{equation*}
\sup_{x \in \mathbb{R} ^d}  \eval{\norm{\vec{D}_x ^i  X_t ^x }^p }< \infty.
\end{equation*}
\end{theorem}
\begin{assumption}\label{ass:phiC2b}
Let the test function $\phi: \mathbb{R}^d \to \mathbb{R}$ and $\phi\in C^{2} _{b}(\mathbb{R}^d)$.
\end{assumption}
Before continuing we define the quantity
\begin{equation}  \label{eq:psidef}
    \Psi(t,x) :=\eval{\phi(X_t) \vert X_0=x }=\eval{\phi(X_t^x)}.
\end{equation}
\begin{theorem}[\cite{cerrai}, Theorem 1.6.2] \label{teo:KolmogPDE}
Let Assumption \ref{ass:1} with $H=3$, Assumption \ref{ass:2} and Assumption \ref{ass:phiC2b} hold. Let $X_t$ be the solution to \eqref{eq:EqAFB}.
Then, $\Psi(t,x)$ defined in \eqref{eq:psidef} is the unique classical solution to the 
Kolmogorov PDE 
\begin{equation} \label{eq:KolmogorovPDE}
\der{}{t} \Psi(t,x)=\LL \Psi(t,x)
\end{equation}
where, with $\mu$ and $\sigma_i$ defined in \eqref{eq:musigmadefn}, $\LL$ is given by
$$
 \LL \Psi(t, x)  :=\vec{D} \Psi(t,x) \mu(x)  +\frac{1}{2} \sum_{i=1}^ m \sigma_i(x)^\intercal   \vec{D}^2  \Psi(t,x)  \sigma_i(x).
$$
\end{theorem}
\begin{remark}
    Rather than imposing an extra hypotheses (Hypotheses 1.3) as in \cite{cerrai} on the diffusion coefficient we assume in Assumption \ref{ass:phiC2b} that $\phi$ is in  $C^2 _b (\mathbb{R}^d)$.
    Together with the mean-square differentiability of $X_t^x$ given in Theorem \ref{teo:bounded_derivatives}, $\Psi(t,x)$ then satisfies smoothness and boundedness conditions required in the proof of Theorem 1.6.2 in \cite{cerrai}. 
\end{remark}     

\subsection{Tamed GBM Method and Convergence Results}
Our class of exponential methods takes advantage of the linear terms in \eqref{eq:EqAFB} by exploiting the stochastic operator
\begin{equation} \label{eq:semigroup}
\SG{t,t_0}=\exp \left(( A-\frac{1}{2} \sum_{i=1} ^m B_i ^2)(t-t_0) +
  \sB (W_t^i  -W_{t_0}^i) \right)
\end{equation}
which, under the commutativity condition \eqref{eq:commute}, is the  solution   to 
\begin{equation} \label{eq:Homogen}
\d \SG{t, t_0} = A  \SG{t,t_0}   \dt + \sum _{i=1} ^m  \ B_i \SG{t,t_0} \dW_t^i, \qquad \SG{t_0, t_0} =I_d.
\end{equation}
Given $N\in \mathbb{N}$ and final time $T$ we set the time step size $\Delta t=\frac{T}{N}$. This gives the uniform time partition $0=t_0<t_1    <t_2<\hdots <t_N=T$ with $t_n=n\Dt$. We denote increments $\DW{n}:=W^i_{t_{n+1}}  -W^i_{t_n}$.

We propose and prove weak convergence of the tamed GBM  method
\begin{equation}  \label{eq:tamed_EI0}
Y_{n+1}^N=\SG{t_{n+1},t_n}\left(Y_n^N+ \! \left( \Ft(Y_n^N)-\sB g_i(Y_n^N)\right) \Dt  + \! \sum_{i=1}^m g_i(Y_n^N)  \DW {n}\right)
\end{equation}
where $\Ft$ is the taming term given by
\begin{equation} \label{eq:Ft}
    \Ft(y) := \alpha(\Dt,y)F(y).
\end{equation}
The taming function $\alpha(\Dt,y)$ is assumed to satisfy for all $y\in \mathbb{R}^d$ and $\Dt>0$
\begin{equation}
\label{eq:tamingrequirement}
\norm{\alpha(\Dt,y)F(y)}\Dt \leq 1, \quad 0\leq \alpha(\Dt,y) \leq 1,\quad
\vert \alpha(\Dt,y)-1 \vert \leq C\Dt,
\end{equation}
where $C>0$ is a constant independent of $\Dt$.
The typical form of taming (e.g. \cite{hutzenthaler2012,sabanis2013note}) is to take $\alpha(\Dt,y)=(1+\Dt \|F(y)\|^p)^{-1}$, so that with $p=1$
\begin{equation}\label{eq:tmterm}
\Ft(y)=\frac{F(y)}{1+\Dt \norm{F(y)}}.
\end{equation}
Strong convergence of \eqref{eq:tamed_EI0} with \eqref{eq:tmterm} and $g_i\equiv 0$ was considered in \cite{tamedGBM} and the efficiency of the method was illustrated numerically.
If we take $\alpha(t,y) \equiv 1$, so that $\Ft=F$ in \eqref{eq:tamed_EI0}, then we obtain one of the methods in \cite{utkuLord} (proved to be strongly convergent with order 1/2 under global Lipschitz assumptions). Further it was shown the method is highly efficient for SDE's  with dominant linear terms and, by a homotopy approach, is competitive when applied to highly non-linear forms of \eqref{eq:EqAFB}. Note that it is clear from \eqref{eq:EqAFB} that taking $B_i=0$ and $A=B_i=0$
(or incorporating these terms into the nonlinearities)
we recover from \eqref{eq:tamed_EI0} the exponential tamed and the standard tamed methods respectively.

Before we state our main results we give an integral representation of the continuous version of the numerical method. 
Let us define the continuous extension   of  numerical  solution $\Yt$ for $t\in[t_n ,t_{n+1}]$ by  
\begin{multline}\label{eq:ctsIntegral}
\Yt =  \SG{t,t_n} \Yht+  \SG{t,t_n}\int_{t_n} ^ {t}    \left( \Ft(\Yht) -\sB g_i(\Yht)\right)  \ds \\ + \SG{t,t_n}\int_{t_n} ^ {t} \sum _{i=1}
  ^m  g_i(\Yht)\dW^i_s.
\end{multline}
Then it is  clear that  $Y_n^N=\Yht$ for $t_n \leq t < t_{n+1}$.  
\begin{lemma}\label{lem:itoEq}
Let $ \bar{Y}_t$ be the interpolated continuous version in \eqref{eq:ctsIntegral} of the    numerical  solution given in  \eqref{eq:tamed_EI0}. Then the  differential for this solution is given by
\begin{equation}
\label{eq:cts}
\d\Yt = \mu^{tm} \dt + \sum_{i=1}^m \sigma_i^{tm} \dW_t^i,
\qquad t\in[t_n,t_{n+1}]
\end{equation}
where $\mu^{tm}=\mu^{tm}(t,t_n)$ and $\sigma_i^{tm}=\sigma_i^{tm}(t,t_n)$ and
\begin{equation}\label{eq:mutm_sigmatm_def}
\mu^{tm} := A \Yt + \SG{t,t_n} \Ft (\Yht),    \quad 
\text{and} \quad \sigma_i^{tm} := B_i  \Yt +\SG{t,t_n } g_i(\Yht ).
\end{equation}
\end{lemma}
\begin{proof}
By the definition of the inverse GBM (see for example \cite{kloeden2011}), $$
\d \SG{t,t_n} ^{-1}=\left(-A+\sB^2\right) \SG{t,t_n} ^{-1} \dt - \sB \SG{t,t_n} ^{-1} \dW_t^i.
$$ 
We seek the appropriate $\mut$ and $\sigmat$.
The  product rule for the \Ito differential  gives  
\begin{multline}
\label{eq:tmp3}
\d\left(  \SG{t,t_n} ^{-1} \bar{Y}_t \right)=  \SG{t,t_n} ^{-1}\Big( (-A+\sum _{i=1}
  ^m B_i^2) \bar{Y}_t+\mut-\sum _{i=1}
  ^m B_i\sigmat \Big)  \dt \\ + \SG{t,t_n} ^{-1}\Big(\sum _{i=1}
  ^m\left(\sigmat -B_i \bar{Y}_t \right)
   \Big) \dW_t^i. 
\end{multline}
On the other hand \eqref{eq:ctsIntegral} can be written as
\begin{equation}\label{eq:tmp4}
    \d\left(  \SG{t,t_n} ^{-1} \bar{Y}_t \right) =     \left( \Ft(\Yht) -\sB g_i(\Yht)\right)  \dt+  \sum _{i=1}
  ^m  g_i(\Yht)\dW^i_t.
\end{equation}
By comparison of \eqref{eq:tmp4} with \eqref{eq:tmp3}, we find 
\begin{align}
   \SG{t,t_n} ^{-1}\left((-A+\sum _{i=1}
  ^m B_i^2) \Yt+\mut-\sum _{i=1}
  ^m B_i\sigmat \right) & =  \Ft( \Yht) -\sB g_i( \Yht) \label{eq:first_eq}\\
\SG{t,t_n} ^{-1} \left(\sigmat -B_i \bar{Y}_t \right)& =   g_i (\Yht), \quad i=1,2 \hdots, m. \label{eq:second_eq}
\end{align}
Solving matrix equation \eqref{eq:second_eq} for $ \sigma_i^{tm}$, the commutativity conditions \eqref{eq:commute} and substitution into equation \eqref{eq:first_eq} to determine $\mut$ gives the desired result.
\end{proof}
Our main result is weak convergence of order one of the numerical scheme \eqref{eq:tamed_EI0} and to prove this we make use of bounded moments.
\begin{assumption}\label{ass:global_g} There exists $K>0$ such that for all $i=1,\hdots, m$
 $$\norm{ g_i(y)-g_i(z) }\leq K \norm{y-z},  \quad \forall \ y, z \in  \mathbb{R}^d.$$
\end{assumption}
\begin{remark} 
The global Lipschitz property given in Assumption \ref{ass:global_g} implies boundedness of $\vec{D}g$. Together with Assumption \ref{ass:1}, by the mean value Theorem, there exists $K \in \mathbb{R}$ such that for all $y,z \in \mathbb{R}^d$
\begin{equation}\label{eq:one_sided_ineq}
 \innerproduct{y-z}{F(y)-F(z)}\leq K \norm{y-z}^2.
 \end{equation}
\end{remark}
\begin{theorem}\label{teo:boundedMoment}
Let Assumption \ref{ass:1} with $H=1$ and Assumption  \ref{ass:global_g} hold. Then, for  $Y_n^N$ be given by \eqref{eq:tamed_EI0}  and for  all  $p \in [1,\infty)$
$$
\sup_{N \in \mathbb{N}}  \sup_{0\leq n\leq N}   \eval{ \norm{Y_n^N}^p }<\infty.
$$
\end{theorem}
The proof of this Theorem is given in Section \ref{sec:BoundedMoments} and follows the approach of \cite{hutzenthaler2012}. 
However we need to control the stochastic operator in \eqref{eq:semigroup} and in contrast to \cite{tamedGBM} we also now need to take account of the nonlinear diffusion terms $g_i$. The main novelty in our proof below is the interaction between these two terms. 

\begin{theorem}
\label{thrm:1}
Let Assumption \ref{ass:1} with $H=4$, Assumption \ref{ass:2} and Assumption \ref{ass:global_g} hold. Let $X_T$ be the solution to \eqref{eq:EqAFB}. Let $Y_N^N$ be found from \eqref{eq:tamed_EI0} and \eqref{eq:tamingrequirement} hold.
Then, for all $\phi:\mathbb{R}^d\to\mathbb{R}$, $ \phi \in  C^{4} _{b} (\mathbb{R}^d)$ there is a constant $C>0$, independent of $\Dt$ such that 
$$
\vert \eval{  \phi(Y_N^N ) }-\eval{\phi(X_T )}   \vert \leq
C \Delta t.
$$
\end{theorem}
Although in Theorem \ref{teo:KolmogPDE} we have $\phi \in  C^2 _b (\mathbb{R}^d)$, to get first order weak convergence, we  impose  $\phi \in  C^4 _b (\mathbb{R}^d)$ in Theorem \ref{thrm:1}.
 We prove Theorem \ref{thrm:1} in Section \ref{sec:WeakConvergence} using the Kolmogorov equation. Once again 
we need to take careful account of the stochastic operator $\SG{t,t_0}$ from \eqref{eq:semigroup} as well as dealing with the one-sided Lipschitz drift $F$.  Before giving the proofs we present some numerical results.

\section{Numerical Results}
\label{sec:NumericalResults}

We seek to estimate numerically the weak discretization error $
\vert \expect{\phi(X_T)}-\expect{\phi(Y_N^N)}\vert$, where $Y_N^N$ is a numerical approximation to $X_T$ with $\Dt= T/N$.  To illustrate the rate of convergence we need to estimate the weak error for different values of $\Dt$.
Our aim is to examine this in the absence of an analytic solution or where the numerical solution of the Kolomogorov equation is prohibitively expensive. 
We also wish to illustrate the weak convergence rate of order 1 that is proved in Theorem \ref{thrm:1}.

In practice we take a reference numerical solution so that $X_T\approx \Y^\Lref$ with $\Lref>>N$ (with $\Dt_R=T/N_R$). 
We then estimate
$$\vert\expect{\phi(\Y^\Lref)}-\expect{\phi(Y_N^N)}\vert=
\vert \expect{\phi(\Y^\Lref)-\phi(Y_N^N)} \vert.
$$
Note that $\Y^\Lref$ may be computed by a different method to that for $Y_N^N$. In \cite{lang2018} issues in computing weak errors using MLMC methods for SPDEs are discussed with multiplicative noise and upper and lower bounds of simulation errors are obtained. However the authors did not consider the simultaneous computation of a reference solution. In \cite{abladinger2017} the MLMC method is examined where the zero solution is asymptotically mean square stable and an importance sampling technique was introduced.  We observe similar stability issues below. 

We briefly discuss four approaches to estimate the weak error using the 
multi-level Monte-Carlo technique (MLMC), see \cite{giles_2015,Lord20141}.
We denote these methods \textbf{Trad}, \textbf{MLMCL0}, \textbf{MLMC}, \textbf{MLMCSR} and examine them numerically in our experiments.
In a traditional method, denoted \textbf{Trad},  we estimate independently $\expect{\phi(\Y^\Lref)}$ and $\expect{\phi(Y_N^N)}$ by a MLMC method.
Thus, for the reference solution we have
\begin{equation}
\label{eq:MLMCref}
    \expect{\phi(\Y^\Lref)}  = \expect{\phi(\Y^{N_0})} + \sum_{\ell=1}^{R} \expect{\phi(\Y^{N_{\ell}}) - \phi(\Y^{N_{\ell-1}})}
\end{equation}
and for the approximation, with $N_L:=N$ 
\begin{equation}
\label{eq:MLMCapprx}
    \expect{\phi(Y_N^N)} = \expect{\phi(Y_{N_0}^{N_0})} + \sum_{\ell=1}^{L} \expect{\phi(Y_{N_\ell}^{N_\ell}) - \phi(Y_{N_{\ell-1}}^{N_{\ell-1}})}.
\end{equation}
An alternative is to exploit difference in the telescoping sums from the MLMC approach for the reference $\Y^\Lref$ \eqref{eq:MLMCref} and numerical approximation $Y_N^N$ \eqref{eq:MLMCapprx}. Subtracting we get
    \begin{align}
      \expect{\phi(\Y^\Lref)} & - \expect{\phi(Y_N^N)}    = \expect{\phi(\Y^{N_0})-\phi(Y^{N_0}_{N_0})} \nonumber \\
        & + \sum_{\ell=1}^{L} \expect{\phi(\Y^{N_\ell}) - \phi(\Y^{N_{\ell-1}}) - [\phi(Y_{N_\ell}^{N_\ell}) - \phi(Y_{N_{\ell-1}}^{N_{\ell-1}})]} \nonumber  \\
    & + \sum_{\ell=L+1}^{R} \expect{\phi(\Y^{N_{\ell}}) - \phi(\Y^{N_{\ell-1}})}. 
    \label{eq:MLMCerror}
    \end{align}
For the coarsest level $\ell=0$ we have a choice of estimating 
$\expect{\phi(\Y^{N_0})}-\expect{\phi(Y_{N_0}^{N_0})}$ or
$\expect{\phi(\Y^{N_0})-\phi(Y_{N_0}^{N_0})}$. If the reference $\Y^{N_0}$ is found with a different method to that of $Y^{N_0}_{N_0}$, we can expect some variance reduction using the latter method over the former (and so fewer samples required to approximate the expectation). Note that if only computing for a fixed single $L$ then we expect further variance reduction for the second term, i.e. in estimating 
$$\expect{\phi(\Y^{N_\ell}) - \phi(\Y^{N_{\ell-1}}) - [\phi(Y_{N_\ell}^{N_\ell}) - \phi(Y_{N_{\ell-1}}^{N_{\ell-1}})]}.$$
However, here we wish to compute for different values of $L$ in order to illustrate the rate of convergence. Thus, rather than recomputing \eqref{eq:MLMCerror} for different $L$ we instead as follows
(to avoid recomputing the MLMC estimates).
    
For \textbf{MLMCL0} we exploit efficiency of variance reduction on the coarsest level and so estimate
$\expect{\phi(\Y^{N_0})-\phi(Y^{N_0}_{N_0})},$ with
estimations of
\begin{equation*}\label{eq:MLMC0tmp}
    \expect{\phi(Y_{N_\ell}^{N_\ell}) - \phi(Y_{N_{\ell-1}}^{N_{\ell-1}})},\quad \ell\in\mathbb{N}^+_L, \quad \text{\&}
\quad 
\expect{\phi(\Y^{N_{\ell}}) - \phi(\Y^{N_{\ell-1}})}, \quad  \ell\in\mathbb{N}^+_{R}.
\end{equation*}
Where as for \textbf{MLMC} we estimate the weak error by \eqref{eq:MLMCerror} and estimate
$\expect{\phi(\Y^{N_0})}$, $\expect{\phi(Y^{N_0}_{N_0})}$ separately.
Finally we see from \eqref{eq:MLMCerror} that if the same numerical scheme is used to estimate both the reference solutions and $Y_N^N$ that
 \begin{align*}
      \expect{\phi(\Y^\Lref)}- \expect{\phi(Y_N^N)}   & =  \sum_{\ell=L+1}^{R} \expect{\phi(\Y^{N_{\ell}}) - \phi(\Y^{N_{\ell-1}})}. 
    \label{eq:MLMCerrorTail}
    \end{align*}
    This method does not allow for a more accurate (e.g. higher order) reference solution (as it uses a self reference).
    We call this {\bf MLMCSR} and estimate 
$$\expect{\phi(\Y^{N_{\ell}}) - \phi(\Y^{N_{\ell-1}})}, \qquad \ell=1,2,\ldots,R.$$

To illustrate the rate one weak  convergence results of Theorem \ref{thrm:1} we consider the following cubic SDE in $\mathbb{R}^d$
\begin{equation*}\label{eq:ACexample}
    \d X = [AX + X-X^3] \dt + \beta_1 X \dW + \beta_2 \frac{X}{1+X^2} \dW.
\end{equation*}
This cubic equation is often used as a test equation as solutions exhibit transitions between different phases. With $d=1$ it is sometimes known as the Ginzburg-Landau equation, \cite{kloeden2011}. For larger $d$ the system of SDEs can be thought of as arising from the spatial discretization of a stochastic Allen-Cahn equation, see for example \cite{lord2004,utkuLord,hutzenthaler2011,brehier2020weak,cai2021weak}. 
We first examine the linear diffusion with $\beta_1=0.1$, $\beta_2=0$ and then non-linear diffusion $\beta_1=\beta_2=0.1$. We look at dimensions $d=1,4,10,50$ and $d=100$. 
For $d=1$ we take $A=-4$, $X_0=x=0.5$. For $d\geq 4$
$A$ is the standard tridiagonal matrix from the finite difference approximation to the Laplacian:
$
A=0.5 d^{-2}\text{diag}(1,-2,1)
$
and as initial data we take 
$X_0=x=0.5\exp(-10(y-0.5)^2)$ where 
$y=[1/d,2/d,\ldots,(d-1)/d]^T$.
We solve to $T=1$ and as $\phi$, we take $\phi(y)=\|y\|^2$.
The standard exponential tamed method is used as the reference solution $\Y^\Lref$ (except for \textbf{MLMCSR} when looking at our GBM tamed method) and we take \eqref{eq:tmterm} as our taming function.
We perform 10 separate runs computing the full weak error convergence plot and use this to estimate the standard deviation.
%

\begin{table}[]
    \centering
    \begin{tabular}{|c|c|c|c|c|}\hline
    Method & $d=1$ & $d=4$ & $d=10$\\ \hline
         \textbf{MLMCL0} &  18.51 (9.25) & 
         180.87 (9.89) & 303.31 (0.47) \\ \hline
         \textbf{MLMC}   & 15.51. (0.08) & 
         177.92 (0.41) & 304.07 (0.39) \\ \hline
         \textbf{Trad} GBM   & 15.30 (13.82) &  
         141.57 (0.29) & 247.92 (0.05) \\ \hline
         \textbf{Trad} Tamed  & 13.82 (0.02)  & 
         81.10 (0.23)
         & 151.65 (0.21) \\ \hline
    \end{tabular}
    \caption{Average computational times in seconds (and standard deviation) required to compute the convergence plots (e.g. Fig. \ref{fig:d1} (a) or (b) or one curve in (c) or (d)). Note that \textbf{MLMCSR} is computed directly from \textbf{MLMC}.}
    \label{tab:timings}
\end{table}
%


First we examine the linear noise case with $\beta_1=0.1$ (and $\beta_2=0$). 
Table \ref{tab:timings} compares the efficiency of the different methods for estimating the weak error for dimensions $d=1,4$ and $10$. We restrict $\Dt<0.0075$ in these computations to ensure we see convergence of the exponential tamed method.
We observe, as expected, that estimating the weak error directly by either \textbf{MLMCL0} or \textbf{MLMC} is far more efficient than using the MLMC in a traditional from \textbf{Trad}. We also observe there is a slight advantage to using \textbf{MLMCL0} compared to \textbf{MLMC} (due to some variance reduction on the coarsest level). \textbf{Trad} (GBM)  looks at convergence for the GBM based methods and (Tamed) to the exponential tamed method, we observe a small overhead in the GBM method compared to Tamed.
However, as we discuss below, there are advantages to the GBM approach.  An alternative method would be to take a drift implicit method and to solve a nonlinear system of equations at each step. 
The weak errors are similar to those for GBM.
Using the standard \texttt{fsolve} (without providing the linearization) in MATLAB leads to 
computational times two orders of magnitude larger (e.g. 2396 seconds (8.6) for MLMC0 and 2396 seconds (2.8) for MLMC) performed on the same cluster. For this reason we only show convergence plots for the explicit methods.

In Fig. \ref{fig:d1} we show weak convergence for $d=1$ and in all approaches observe the expected rate of convergence one. Furthermore in (c) and (d), where we have a direct comparison to the exponential tamed method, we see the GBM based method has a smaller error constant. We see in (d) that the two smallest step sizes ($\Dt=2^{-9}$ and $\Dt=2^{-9}$) are too close to the reference time step to observe convergence.
\begin{figure}[h]
\centering
\begin{subfigure}[t]{0.45\textwidth}
\includegraphics[width=\textwidth]{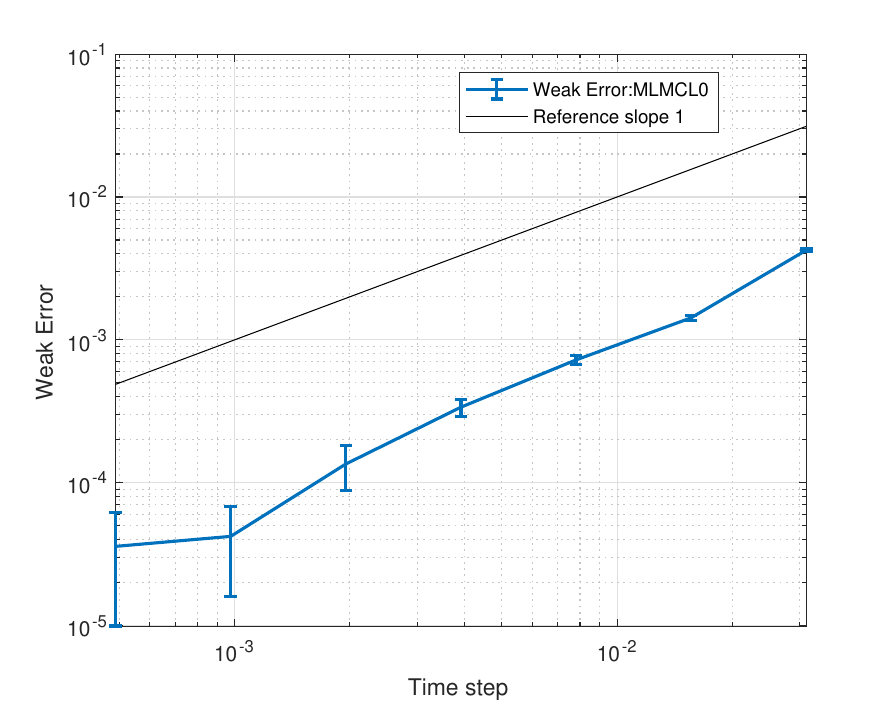}
\caption{\textbf{MLMCL0}}
\end{subfigure}
\begin{subfigure}[t]{0.45\textwidth}
\includegraphics[width=\textwidth]{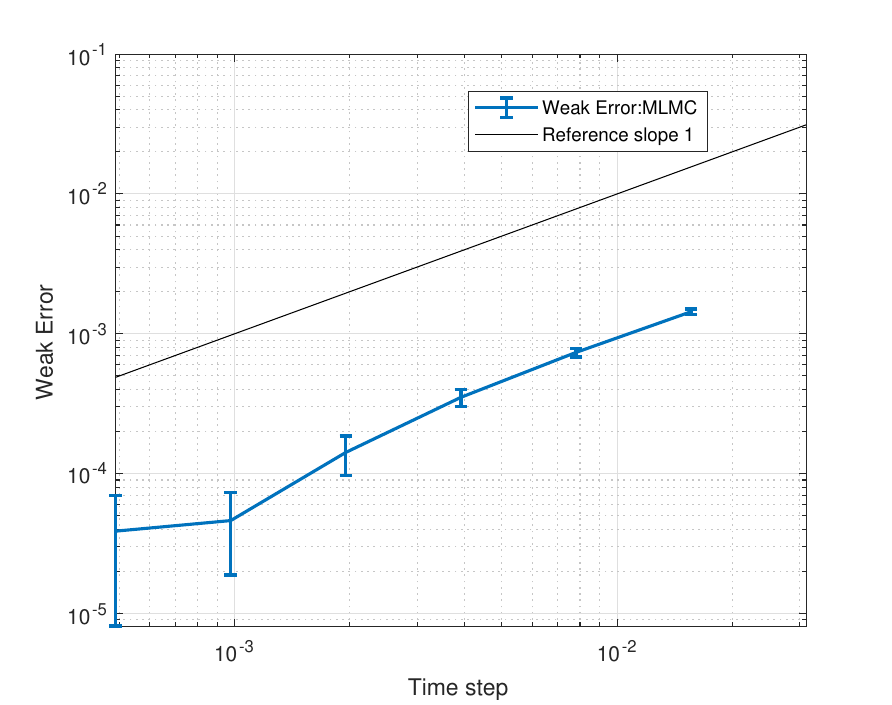}
\caption{\textbf{MLMC}}
\end{subfigure}
\begin{subfigure}[t]{0.45\textwidth}
\includegraphics[width=\textwidth]{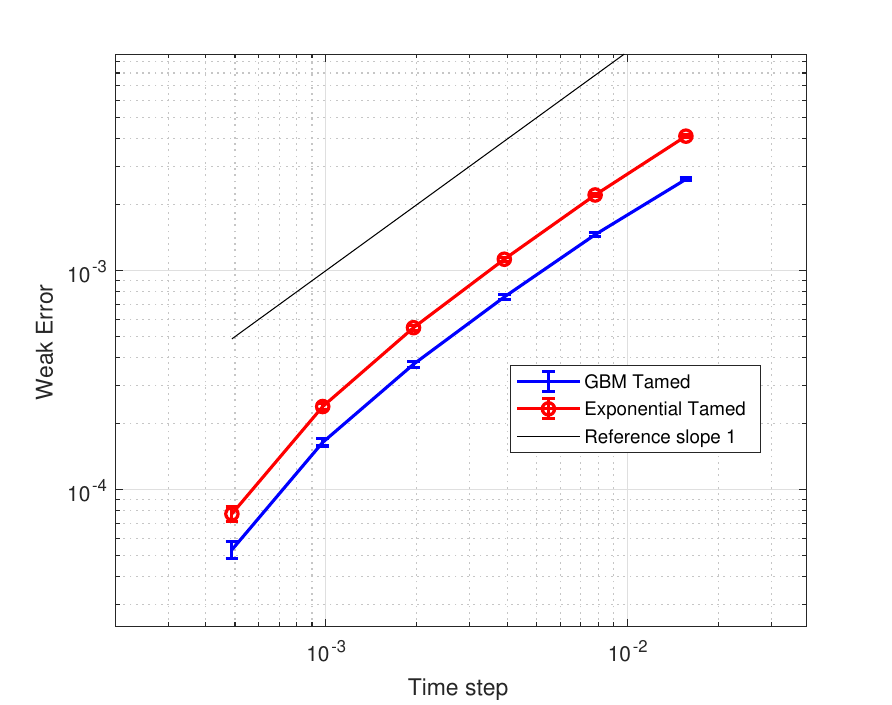}
\caption{\textbf{MLMCSR}}
\end{subfigure}
\begin{subfigure}[t]{0.45\textwidth}
  \includegraphics[width=\textwidth]{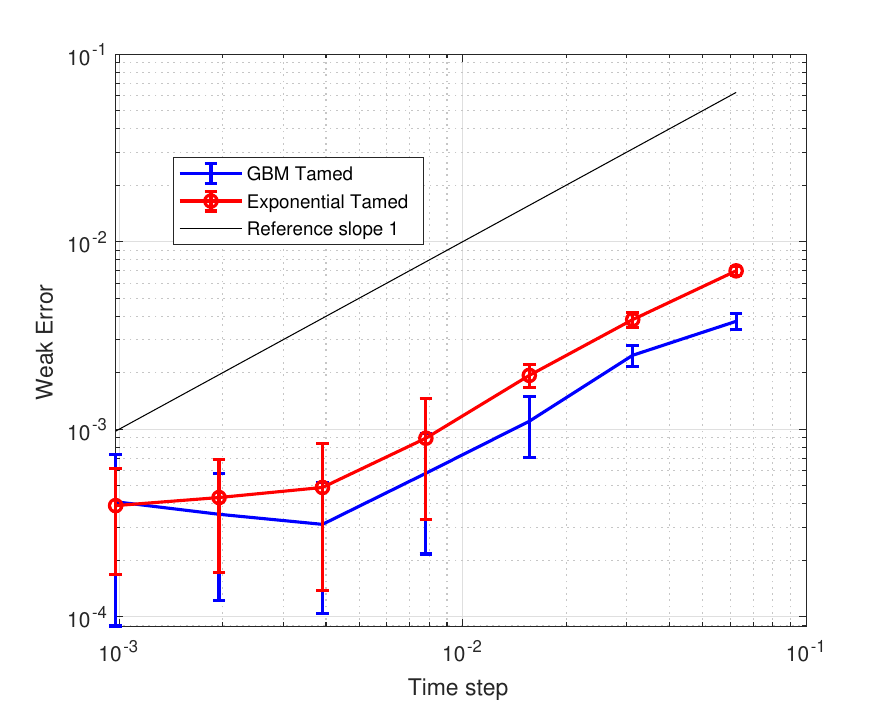}
\caption{\textbf{Trad}}
\end{subfigure}
\caption{Comparison of different approaches to estimate the weak error for $d=1$ with $\beta_1=0.1$ and $\beta_2=0$.}\label{fig:d1}
\end{figure}

In Fig. \ref{fig:d4} we show convergence for $d=4$ and observe the predicted convergence of rate 1. 
In (a), (b) and (c) we restrict the largest step so that $\Dt<0.02$. For the exponential tamed method it was essential to impose this restriction on the maximum step size as, although solutions remain bounded, the error is too large to observe convergence. These large solutions also lead to large variances and hence an infeasible large number of samples, (see also the discussion in \cite{abladinger2017}). 
This is illustrated by comparing (c) where we restrict $\Dt<0.02$ and (d) where $\Dt<0.2$. The exponential tamed method only starts to converge for $\Dt\leq 0.02$.
For the larger step sizes in (d) we see the GBM based method performs well (unlike the exponential tamed method). The small step size restriction on the standard exponential method as $d$ increases makes it difficult to obtain a reference solution using this method using the \textbf{MLMC} method.

\begin{figure}[h]
\centering
\begin{subfigure}[t]{0.45\textwidth}
\includegraphics[width=\textwidth]{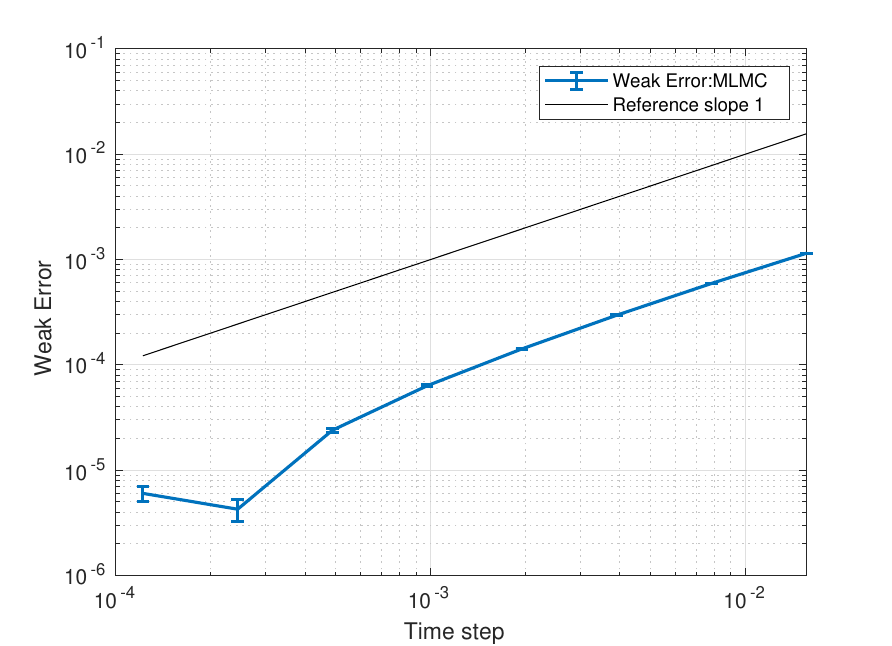}
\caption{\textbf{MLMCL0}}
\end{subfigure}
\begin{subfigure}[t]{0.45\textwidth}
\includegraphics[width=\textwidth]{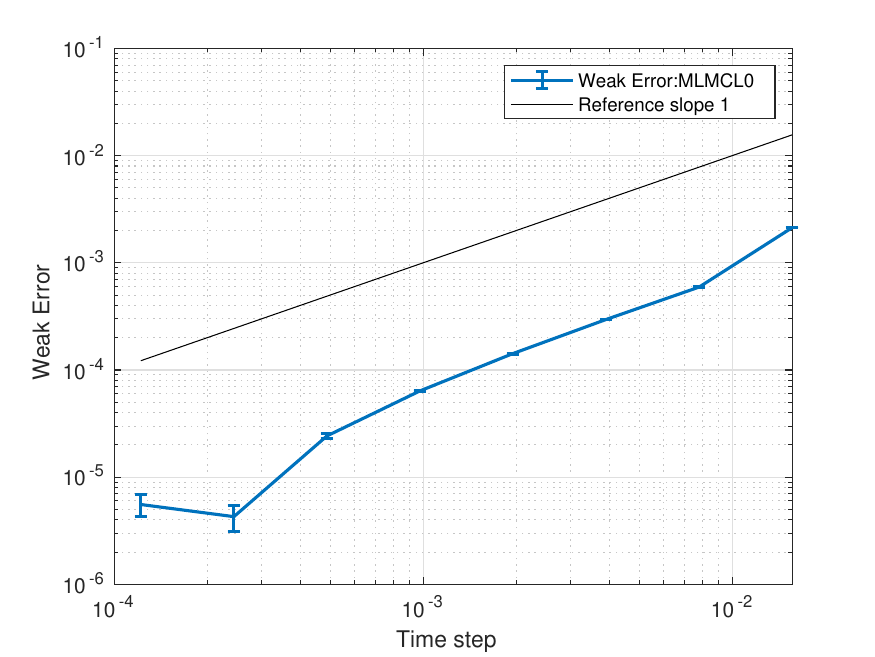}
\caption{\textbf{MLMC}}
\end{subfigure}
\begin{subfigure}[t]{0.45\textwidth}
\includegraphics[width=\textwidth]{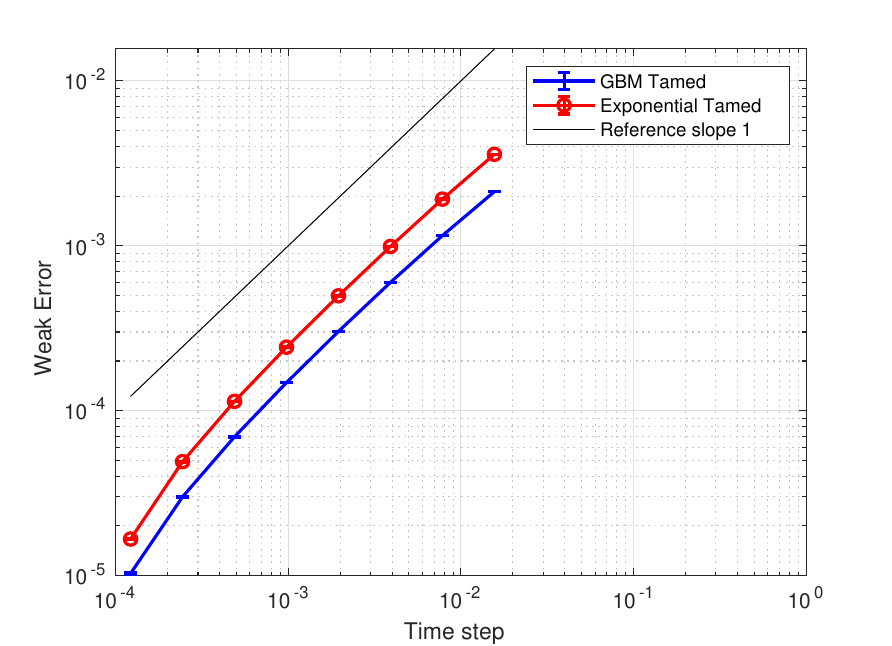}
\caption{\textbf{MLMCSR} ($\Dt<0.02$)}
\end{subfigure}
\begin{subfigure}[t]{0.45\textwidth}
\includegraphics[width=\textwidth]{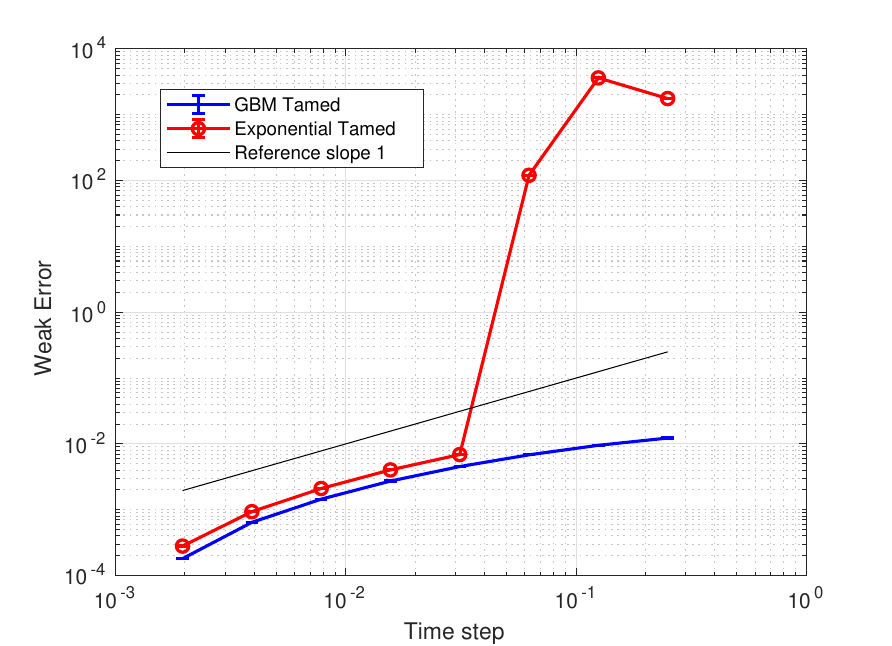}
\caption{\textbf{MLMCSR} ($\Dt<0.2$)}
\end{subfigure}
\caption{Comparison of 3 different approaches to estimate the weak error for $d=4$ with $\beta_1=0.1$ and $\beta_2=0$. In (d) we see that the exponential tamed method is not well behaved for large time step size.}\label{fig:d4}
\end{figure}

In Fig. \ref{fig:d50d100} we illustrate convergence using \textbf{MLMCSR} for the two methods with $d=50$ and $d=100$. We only examine \textbf{MLMCSR} as for these larger values of $d$ due to unreliable results the exponential tamed method for larger $\Dt$ values (as discussed for Fig. \ref{fig:d4}). We see that for $d=50$ we require $\Dt<2\times 10^{-4}$ and for $d=100$ that $\Dt<10^{-4}$ to obtain an approximation using the standard exponential tamed method.
However, we observe that the GBM method converges with the predicted order and there is no issue with large variances.

\begin{figure}[h]
\centering
\begin{subfigure}[t]{0.45\textwidth}
\includegraphics[width=\textwidth]{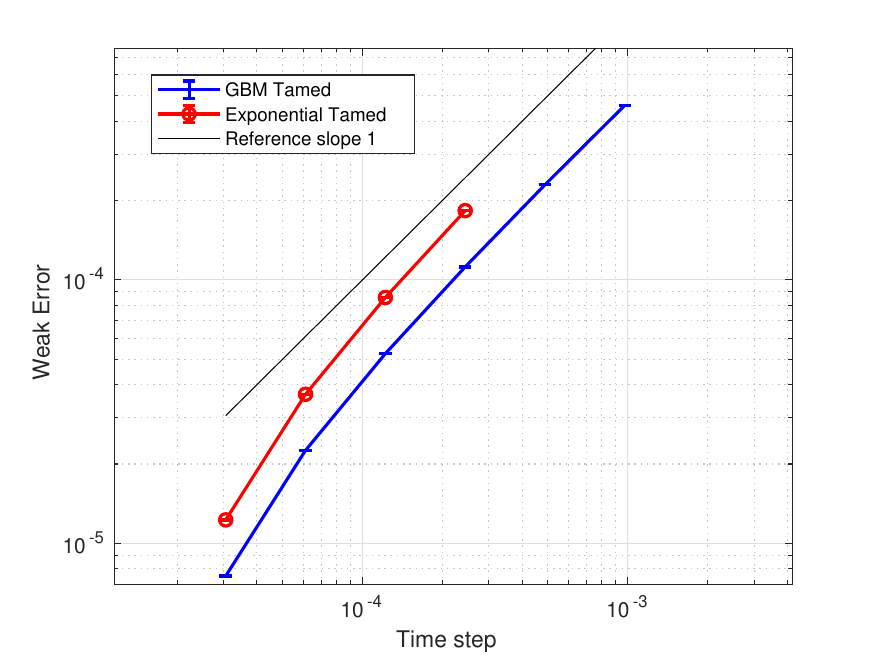}
\caption{\textbf{MLMCSR} $d=50$}
\end{subfigure}
\begin{subfigure}[t]{0.45\textwidth}
\includegraphics[width=\textwidth]{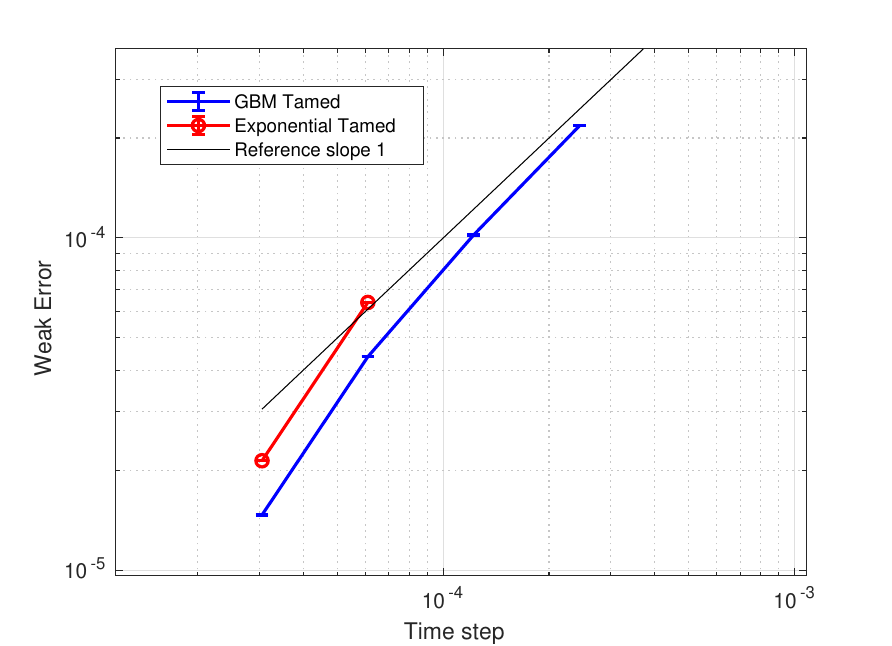}
\caption{\textbf{MLMCSR} $d=100$}
\end{subfigure}
\caption{Convergence for $d=50$ and $d=100$ with $\beta_1=0.1$ and $\beta_2=0$. For larger time step sizes the exponential tamed method has a large variance and we do not see convergence until smaller time step sizes.}\label{fig:d50d100}
\end{figure}

For the nonlinear diffusion, with $\beta_1=\beta_2=0.1$ in \eqref{eq:ACexample}, we illustrate convergence in Figure \ref{fig:d4nonlin} (a) \textbf{MLMCSR} and (b) \textbf{Trad}. 
We have taken $d=4$ and observe weak convergence of rate one. Both the methods \textbf{MLMC} and \textbf{MLMCL0} also show rate one convergence with the GBM based method having a smaller error constant.
For this nonlinear noise, for larger values of $\Dt$, although the taming ensures solutions remain bounded the variance does not reduce. This is now for both methods. 

\begin{figure}[h]
\centering
\begin{subfigure}[t]{0.45\textwidth}
\includegraphics[width=\textwidth]{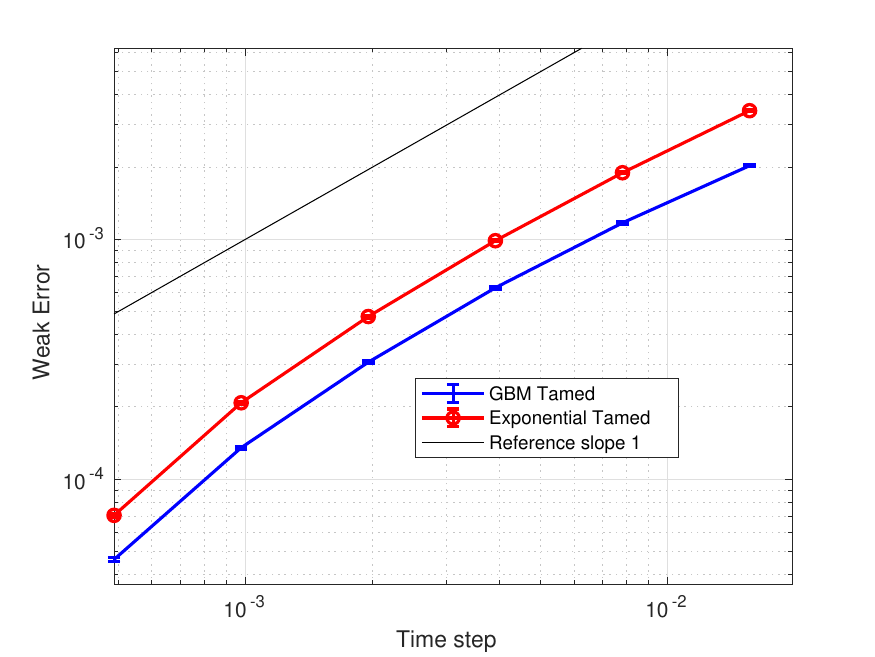}
\caption{\textbf{MLMCSR}}
\end{subfigure}
\begin{subfigure}[t]{0.45\textwidth}
\includegraphics[width=\textwidth]{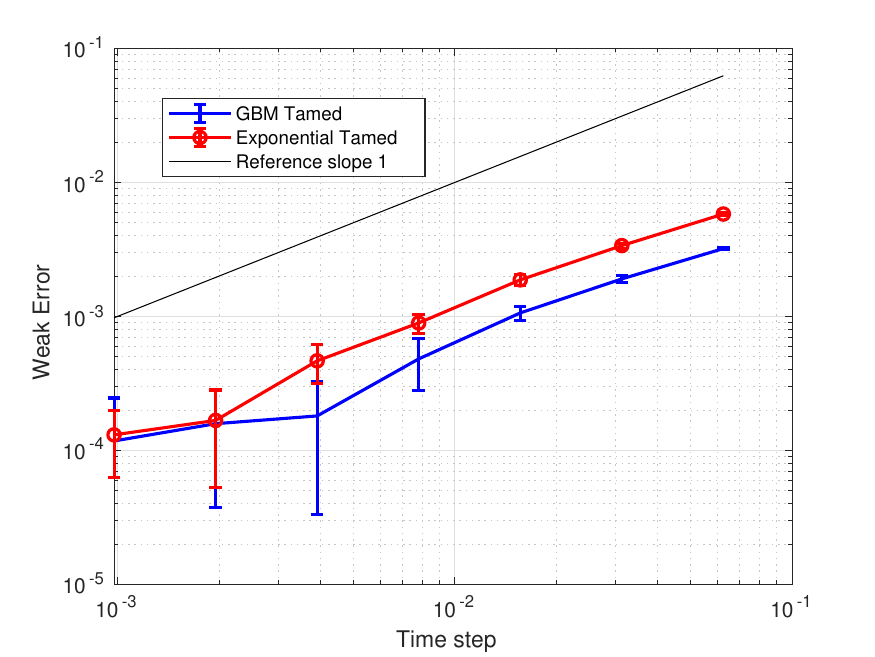}
\caption{\textbf{Trad}}
\end{subfigure}
\caption{Weak convergence for $d=4$ with $\beta_1=\beta_2=0.1$. 
We observe rate one convergence. A time step size restriction is required.}\label{fig:d4nonlin}
\end{figure}

\section{Proof of Theorem \ref{teo:boundedMoment} : bounded moments}
\label{sec:BoundedMoments}

We let $
\Delta \vec{W}_{k}:=(\Delta W_k^1,\Delta W_k^2,\ldots,\Delta W_k^m)^T.$
For the drift we define the notation
\begin{equation}  
\label{eq:ftilde}
\FG (x) := \Ft(x) -\sB g_i(x) 
\end{equation} 
where $\Ft$ is defined in \eqref{eq:Ft}.
We can use $\FG$ to re-write \eqref{eq:tamed_EI0} as
\begin{equation*}  \label{eq:tamed_EI0_short}
Y_{n+1}^N=\SG{t_{n+1},t_n}\left( Y_n^N+  \FG(Y_n^N) \Dt  + \sum_{i=1}^m g_i(Y_n^N)  \DW {n} \right).
\end{equation*}
We  now prove the boundedness of the operator $\SG{t,t_0}$ in  stochastic $L^p$ spaces.

\begin{lemma} \label{lem:SG_bnd}
Suppose $p \geq 2$, $0\leq s<t\leq T $ and let $\SG {t,s}$ be given in \eqref{eq:semigroup} with $t_0=s$. 
Then,
\begin{itemize}
    \item[i)]
For any $\mathcal{F}_{s}$ measurable random variable $v$ in  $L^p (\Omega,\mathbb{R}^d )$
  $$ \Lpnorm{\SG {t,s} v}  \leq    \exp{ \left(\left(\norm{A}+\frac{p-1}{2}\sum _{i=1}
  ^m  \norm{ B_i}^2\right)   (t-s)\right)  }  \Lpnorm{ v}.$$ 

\item[ii)]
There are constants $\kappa_{1,2}>0$, independent of $(t-s)$ such that 
  $$ \eval{\norm{\SG {t,s}}^p } \leq \kappa_1\exp(\kappa_2 (t-s)).$$  
  \end{itemize}

\end{lemma}
\begin{proof}
For the proof of $i)$, see  \cite{tamedGBM}. 

To show $ii)$ we recall that for a standard Gaussian variable $z \sim N(0,1)$ and $\alpha>0$
\begin{equation} \label{eq:abs_distribition}
\eval{\exp(\alpha \vert z \vert )}=\dfrac{1}{\sqrt{2 \pi}}  \int_{-\infty} ^{\infty} \exp(\alpha \vert x \vert ) \exp(-x^2/2)dx
=\exp(\alpha^2/2) \left( 1+\text{erf} (\alpha/\sqrt{2})\right ).
\end{equation}
From \eqref{eq:semigroup} we have that 
\begin{equation} \label{eq:exp_ineq}
\norm{\SG {t,s}} ^p \leq  \exp\left( \norm{ p( A-\frac{1}{2} \sum_{i=1} ^m B_i ^2)(t-s) } \right) \exp \left( 
  p  \sum_{i=1}^m \norm{ B_i} \vert W_t^i  -W_{s}^i  \vert  \right). 
\end{equation}
By taking  expectation of both sides in \eqref{eq:exp_ineq} and considering the  independence and distribution of  random  values $\vert W_t^i  -W_{s}^i  \vert$ and using 
\eqref{eq:abs_distribition} \, for $i=1,2,\hdots, m$, we have
\begin{eqnarray*}
\lefteqn{\eval{\norm{\SG {t,s}}} ^p } \\
& \leq & \exp\left( \norm{ p( A-\frac{1}{2} \sum_{i=1} ^m B_i ^2)(t-s) } \right)  \prod_{i=1}^m  \eval{ \exp \left( 
  p  \norm{ B_i} \vert W_t^i  -W_{s}^i  \vert  \right) } \\
&  = &\exp\left( \norm{ p( A-\frac{1}{2} \sum_{i=1} ^m B_i ^2)(t-s) } \right) \\
& & \times \prod_{i=1}^m  \exp\left(\dfrac{1}{2}p^2 \norm{B_i}^2 (t-s)\right) \left(1+\text{erf}(\frac{p \sqrt{t-s} \norm{B_i}}{\sqrt {2} }) \right).
\end{eqnarray*}
By the boundedness of the function $\text{erf}$, the positive  constants $\kappa_1$  and $\kappa_2$ are determined in terms of $p, \norm{A}$ and $\norm{B_i}$.
\end{proof}
For  the  proof of Theorem \ref{teo:boundedMoment}, we adapt  the approach given in \cite{hutzenthaler2012}.
In \cite{tamedGBM} a similar approach was taken with linear diffusion and so the additional element below is dealing with the nonlinear diffusion terms $g_i$.
We start 
by introducing appropriate sub events of $\Omega$. 
We let $\Omega_0^N:=\Omega$, then for $n\in \mathbb{N}_N ^+$
\begin{equation*}\label{eq:Omega}
\Omega_n^N:=\left\{\omega \in \Omega \vert \sup _{ k \in \NZ{n-1} } D_k^N(\omega) \leq N^{1/4r}
,\sup _{ k \in \NZ{n-1}} \norm{\Delta \vec{W}_{k}}\leq 1 \right\},
\end{equation*}
where the parameter $r$ is as defined in Assumption \ref{ass:1}. The dominating stochastic process $D_n^N$ is defined with $D_0^N:=\left( \lambda + \norm{x} \right) e^{\lambda}$ and  for $n \in  \NP{N}$ by
\begin{equation*}\label{eq:Dn}
D_n^N:=(\lambda + \norm{x } )  \sup _{u \in \NZ{n} }  \prod_{k=u }^{n-1 }\norm{\SG{t_{k+1},t_{k}} }  \exp \left( \lambda + \sup _{u \in \NZ{n} } \sum_{k=u} ^{n-1} \left( \lambda  \norm{\Delta \vec{W}_k}^2  +\beta_k ^N \right) \right) 
\end{equation*}
where
\begin{equation*} \label{eq:alpha}
  \beta_k ^N:= \mathbbm{1}_{\norm{Y_k^N} \geq 1}  \frac{{\innerproduct{Y_k^N} {\sg(Y_k^N)\DW {k}}} }{\norm{Y_k^N}^2} 
  \quad \text{and} \quad
\lambda:=\max\{\lambda_0,\lambda_1,\lambda_2,\lambda_3,\lambda_4\}.
\end{equation*} 
Here $\lambda_0,\ldots,\lambda_4$ are constants defined by
\begin{align*}
\lambda&_0  :=\exp\left({\sum_{i=1}^m \norm{B_i}+T \norm{ A-\frac{1}{2} \sum\limits_{i=1} ^m B_i ^2}}\right) \times \bigg( 1+2TK+T \norm{F(0)} \\ 
& \qquad \ldots  +\sumi\norm{B_i} TK+T\norm{\sB g_i(0)}+mK+\sumi \norm{g_i(0) }  \bigg) \\
\lambda_1:&= m\sumi \left(K  +\norm{g_i(0)}  \right)^2 \\
\lambda_2:&=\left(2K+\norm{F(0)}   \right)^4 \\
\lambda_3:&=\left(  \sumi\norm{B_i}K+\norm{\sB g_i(0)}  \right) ^4\\
\lambda_4:&=\left( 4T^2+2T \right)^2,
\end{align*}
where the constant $K>0$ denotes a constant that arises from the constants in Assumption \ref{ass:1}, \ref{ass:global_g} and  \eqref{eq:one_sided_ineq}.
The first result shows we can dominate the numerical solution on the set $\Omega_n^N$.
\begin{lemma} \label{lem:boundedness1}
Let Assumption \ref{ass:1} hold with $H=1$. Let $Y_n^N$ be given by \eqref{eq:tamed_EI0} and let \eqref{eq:tamingrequirement} hold.
For all $n \in \NZ{N}$ we have the pathwise inequality
$$1_{\Omega_n^N} \norm{Y_n^N} \leq D_n^N.$$%
\end{lemma}
\begin{proof}
On $\Omega_{n+1}^N$ by construction $\norm{\vDW{n}}\leq 1$ for $n\in\NZ{N-1}$. Therefore, $\norm{ \SG{t_{n+1},t_n} }< \infty$ for  $n \in \NZ{N-1}$. 
We prove the Lemma on two  subsets of $\Omega_{n+1}^N$
\begin{align*}
S^{(1)}_{n+1} & :=\Omega_{n+1}^N\cap \lbrace \omega \in \Omega \vert \norm{Y_n^N (\omega)} \leq 1 \rbrace\\    
S^{(2)}_{n+1} & :=\Omega_{n+1}^N\cap \lbrace \omega \in \Omega \vert 1 \leq \norm {Y_n^N (\omega)} \leq  N^{1/(4r)} \rbrace.
\end{align*}
First, on  $S^{(1)}_{n+1}$, we have from \eqref{eq:tamed_EI0} and the triangle inequality that
\begin{equation*}
\norm{Y_{n+1}^N} \leq \norm{ \SG{t_{n+1},t_n} } \left( \norm{Y_n^N}+\Delta t \norm{\FG(Y_n^N)}+ \sumi\norm{g_i(Y_n^N)} \norm{\vDW{n}}    \right).
\end{equation*}
Since $\norm{Y_n^N} \leq 1$, $\norm{\vDW{n}}\leq 1$   on $S^{(1)}_{n+1}$, and by the taming inequality $\alpha(\Dt,y)\leq 1 $ from \eqref{eq:tamingrequirement}, we have that 
\begin{equation*}
\norm{Y_{n+1}^N} \leq \norm{ \SG{t_{n+1},t_n} } \left( 1+\Dt \norm{F(Y_n^N)}+\Dt \norm{\sB g_i(Y_n^N)}+\sumi\norm{g_i(Y_n^N)} \right).
\end{equation*}
Adding and subtracting $F(0)$,  $\sB g_i(0)$, as well as $g_i(0)$ for $i=1,\hdots, m$, and then applying the triangle inequality we get 
\begin{align*}
 \norm{Y_{n+1}^N} & \leq  \norm{ \SG{t_{n+1},t_n} } \Big( 1+\Dt \norm{F(Y_n^N)-F(0)}+\Dt \norm{F(0)} \nonumber \\ 
&  +\Dt \norm{\sB g_i(Y_n^N)-\sB g_i(0)}   
 +\Dt\norm{\sB g_i(0)} \\
 & +\sumi \left(\norm{g_i(Y_n^N)-g_i(0)}+\norm{g_i(0)}\right) \Big).
\end{align*}
By the polynomial growth condition on $\vec{D} F$ given in Assumption \ref{ass:1}, the global Lipschitz condition on $g_i$ (Assumption \ref{ass:global_g}), and using that $\Dt\leq T$
\begin{align*}
 \norm{Y_{n+1}^N} \leq  & \norm{\SG{t_{n+1},t_n}} \Big( 1+T  K   \left( 1+\norm{Y_n^N} ^{2r} \right) \norm{Y_n^N} +T \norm{F(0)} \\ 
 & + T K \sum_{i=1}^m \norm{B_i} \norm{Y_n^N} + T \norm{\sB g_i(0)} +mK \norm{Y_n^N}+ \sumi\norm{ g_i(0)} \Big). 
\end{align*}
On  $S^{(1)}_{n+1}$, since $\norm{Y_n^N} \leq 1$  and $ \norm{ \vDW{n}} \leq 1 $, we have
 \begin{equation}\label{eq:Y_bound_S1}
   \norm{Y_{n+1}^N}   \leq  \lambda.
 \end{equation}
To bound 
$S^{(2)}_{n+1}$, we start from \eqref{eq:tamed_EI0} by squaring the norm
\begin{equation*}
\begin{aligned}
&   \norm{Y_{n+1}^N}^2 \leq  \norm{ \SG{t_{n+1},t_n} }^2 \norm{Y_n^N+\Delta t \FG(Y_n^N)+\sg(Y_n^N)\DW{n}}^2  \\
   & \leq \norm{ \SG{t_{n+1},t_n} }^2 \left( \norm{Y_n^N}^2+\Delta t^2 
\norm{ \FG (Y_n^N)}^2 +  m\sumi\norm{g_i(Y_n^N)}^2 \norm{\vDW{n}}^2 \right.  \\
&\left. \! +\! 2 \Dt  \innerproduct{Y_n^N}{ \FG(Y_n^N)} +2  \innerproduct{Y_n^N}{ \sg(Y_n^N)\DW{n} }\! + \!2 \Dt  \innerproduct{\FG(Y_n^N)}{ \sg(Y_n^N)\DW{n} }   \right).
 \end{aligned}\label{eq:tmplem1}
 \end{equation*}
Applying the Cauchy-Schwarz and Arithmetic-Geometric inequalities
\begin{align}  \label{eq:tmplem2}
   \norm{Y_{n+1}^N}^2 &\leq  \norm{ \SG{t_{n+1},t_n} }^2 \left( \norm{Y_n^N}^2+2\Delta t^2 
\norm{ \FG (Y_n^N)}^2 +2 m\sumi \norm{g_i(Y_n^N)}^2 \norm{\vDW{n}}^2 \right. \nonumber \\
&\left.  + 2 \Dt  \innerproduct{Y_n^N}{ \FG(Y_n^N)} +2  \innerproduct{Y_n^N}{ \sg(Y_n^N)\DW {n}}    \right)  .
\end{align}
Let us consider $\FG (Y_n^N)$ (see \eqref{eq:ftilde}).
Following the argument in \cite[(40) in Proof of   Lemma 3.1]{hutzenthaler2012} we have 
$$
\norm{F(Y_n^N)}^2\leq N \sqrt{\lambda} \norm{Y_n^N}^2.
$$
Further, by Assumption \ref{ass:global_g} , the global Lipschitz property of the function $g_i$ 
\begin{equation}\label{eq:Bgi}
\begin{aligned}
   \norm{\sB g_i(Y_n^N)}^2  &  \leq \left(  \sum_{i=1}^m\norm{B_i}K+\norm{\sB g_i(0)}  \right)^2 \norm{Y_n^N}^2   \\
   & \leq  \sqrt{\lambda} \norm{Y_n^N}^2  
\end{aligned}
\end{equation}
and 
$$
\norm{g_i(Y_n^N)}^2\leq  \left( K + \norm{g_i(0)} \right)^2  \norm{Y_n^N}^2 .
$$
By definition of $\lambda$, we have
 $$
 m\sum_{i=1}^m \norm{g_i(Y_n^N)}^2\leq \lambda  \norm{Y_n^N}^2.
 $$
Therefore  the  linear  growth   of  the second term on the RHS of \eqref{eq:tmplem2}
$$
\norm{\FG (Y_n^N)}^2 \leq 4 N  \sqrt{\lambda} \norm{Y_n^N}^2
$$
on  $S^{(2)}_{n+1}$  is obtained.
%
The one-sided Lipschitz condition on $F$ \eqref{eq:one_sided_ineq} 
and the Cauchy–Schwarz inequality give that
\begin{align}\label{eq:F_inner_product_ineq}
\innerproduct{Y_n^N}{F(Y_n^N)} & \leq \innerproduct{Y_n^N}{F(Y_n^N)-F(0)} +\innerproduct{Y_n^N}{F(0)} \nonumber \\
& \leq  \left( K + \norm{F (0)}\right) \norm{Y_n^N} ^2 \nonumber \\ 
& \leq \sqrt{ \lambda } \norm{Y_n^N}^2 
\end{align} 
where we have used that $1\leq \|Y_n^N\|$. 
Therefore by the inequalities \eqref{eq:Bgi}, \eqref{eq:F_inner_product_ineq}, Cauchy-Schwarz inequality for the term $\innerproduct{Y_n^N}{\sum _{i=1}^m B_i g_i (Y_n^N)}$ and $\lambda >1$
\begin{align*}
\innerproduct{Y_n^N}{\FG (Y_n^N)} &= \alpha(\Dt, Y_n^N) \innerproduct{Y_n^N}{F (Y_n^N)}- \innerproduct{Y_n^N}{\sum _{i=1}^m B_i g_i (Y_n^N)}  \\
& \leq \left( \alpha(\Dt, Y_n^N) +1\right) \sqrt{\lambda} \norm{Y_n^N}^2.
\end{align*}
As a result, since  $\Delta t =T/N$, $\alpha(\Dt,Y_n^N) < 1 $, by \eqref{eq:tamingrequirement}, and $(4T^2+2T)\leq \sqrt{\lambda}$, we see that on $S^{(2)}_{n+1}$ \eqref{eq:tmplem2} becomes 
\begin{align} \label{eq:Y_bound_S2}
 \norm{Y_{n+1}^N}^2 & \leq \norm{\SG{t_{n+1},t_n}}^2  \norm{Y_n ^N}^2 \left(1+\frac{8T^2}{N} \sqrt{\lambda} + 2\lambda \norm{\vDW{n}}^2 + \frac{4T}{N}\sqrt{\lambda} +2 \beta_n ^N \right) \nonumber \\
 & \leq  \norm{\SG{t_{n+1},t_n}}^2  \norm{Y_n ^N}^2 \left(1+ \frac{2\lambda}{N} + 2\lambda \norm{\vDW{n}}^2  +2\beta_n ^N \right) \nonumber \\
 & \leq \norm{\SG{t_{n+1},t_n}}^2  \norm{Y_n ^N}^2 \exp \left( \frac{2\lambda}{N} + 2\lambda \norm{\vDW{n}}^2  +2\beta_n ^N \right).
\end{align}
Now  we  carry out an induction argument on $n$. 
The case for  $n=0$ is obvious by initial condition on $\Omega_0^N=\Omega$. Let  $l \in \NZ{N-1}$ and assume $\norm{Y_{n} ^N (\omega)}\leq D_n ^N(\omega)$  holds for all $n \in \NZ{l}$ where  $\omega \in \Omega_n ^N$. We now  prove  that   
 $$\norm{Y_{l+1} ^N(\omega)}\leq D_{l+1} ^N(\omega)
 \quad \text{for all} \quad \omega \in \Omega_{l+1} ^N.$$ 
 For  all   $ \omega \in \Omega_{l+1} ^N$, we have  by the induction hypothesis $\norm{Y_{n} ^N (\omega)}\leq D_n ^N(\omega)\leq N^{1/(4r)}$,  $n\in\mathbb{N}_l$ 
 and $\Omega_{l+1} ^N \subseteq  \Omega_{n+1} ^N$.  
For any $ \omega \in \Omega_{l+1} ^N$,  $\omega$ belongs to $S_{l+1} ^{(1)}$ or $S_{l+1} ^{(2)}$. For the inductive argument we define  a random variable
$$
\tau_l ^N (\omega):=\max \left(  \lbrace  -1 \rbrace \cup \lbrace n \in \mathbb{N}_{l-1} \rbrace \bigg\vert  \norm{Y_n^N (\omega)} \leq 1 \rbrace  \right)
$$ 
%
as in \cite{hutzenthaler2012}.
This definition implies that $1 \leq \norm{Y_n^N (\omega)} \leq N^ {1/(4r)}$ for  all  $n \in \lbrace \tau_{l+1} ^N (\omega)+1, \tau_{l+2} ^N (\omega),\hdots, l \rbrace$. By the bound in  \eqref{eq:Y_bound_S2}
\begin{align*}
&   \norm{Y_{l+1} ^N  (\omega)}  \leq \norm{\SG{t_{l+1},t_l} (\omega)}  \norm{Y_l ^N (\omega)} \exp \left( \frac{\lambda}{N} + \lambda \norm{\Delta \vec{W}_l(\omega)}^2  +\beta_l ^N(\omega) \right)   \\
&  \leq  \norm{Y_{ \tau_{l+1} ^N (\omega)+1} ^N (\omega)} \\
& \times \prod_{n=\tau_{l+1} ^N (\omega)+1 } ^l \left( \norm{\SG{t_{n+1},t_{n}} (\omega)  }  \right) \exp \left( \sum_{n=\tau_{l+1}+1} ^l \left( \frac{\lambda}{N} + \lambda \norm{\Delta \vec{W}_n(\omega)}^2  +\beta_n ^N(\omega) \right) \right)  \\
 &   \leq   \norm{Y_{ \tau_{l+1} ^N (\omega)+1} ^N (\omega)}  \\ 
 & \times \sup _{u  \in \NZ{l+1}  } 
  \prod_{n=u }^l \norm{\SG{t_{n+1},t_{n}} (\omega)  }  \exp \left(\lambda +  \sup _{u \in \NZ{l+1} } \sum_{n=u} ^l \left(  \lambda \norm{\Delta \vec{W}_n(\omega)}^2  +\beta_n ^N(\omega) \right) \right).
\end{align*}
By considering \eqref{eq:Y_bound_S1}, the following completes the induction step and proof 
\begin{align*}
\begin{split}
&\norm{Y_{l+1} ^N  (\omega)} \leq (\lambda + \norm{x} ) \\
& \times \sup _{u \in  \NZ{l+1}  }  \prod_{n=u }^l \norm{\SG{t_{n+1},t_{n}} (\omega)  }  \exp \left( \lambda+ \sup _{u \in  \NZ{l+1}  } \sum_{n=u} ^l \left(  \lambda \norm{\Delta \vec{W}_n(\omega)}^2  +\beta_n ^N(\omega) \right) \right)\\
&=D_{l+1} ^N (\omega).
\end{split}
\end{align*}
\end{proof}

\begin{lemma} \label{lem:Dboundedness} For all $p \geq 1$,
$
\sup_{N \in \mathbb{N} } \eval{\sup_{ n \in \NZ{N} } \vert D_n^N \vert ^p} < \infty .
$
\end{lemma}
\begin{proof}
By two applications of H\"{o}lder's inequalities, we have
\begin{align*}
\begin{split}
&\mathop{\sup_{ N \in \mathbb{N}}}_{N \ge 8 \lambda p T }   \LnormR{\sup_{n \in \NZ{N}} D_n ^N}{p} \\
& \leq e^{\lambda}  \left( \mathop{\sup_{ N \in \mathbb{N}}}_{N \ge 8 \lambda p T }   \LnormR{ \exp  \sum_{k=0} ^{N-1} \left(   \lambda \norm{\Delta \vec{W}_k }^2  \right) }{2p}\right) \\
&\times \sup_{ N \in \mathbb{N}}  \LnormR{ (\lambda + \norm{x } )  \sup_{n \in \NZ{N} } \sup_{u \in \NZ{n} }  \prod_{k=u }^{n-1 }\norm{\SG{t_{k+1}},t_{k}   }    \sup_{n \in \NZ{N} } \exp  \left(\sup_{u \in \NZ{n} }\sum_{k=u} ^{n-1} \ \beta_k ^N  \right)  }{2p} \\
&  \leq e^{\lambda}  \left( \mathop{\sup_{ N \in \mathbb{N}}}_{N \ge 8 \lambda p T }   \LnormR{ \exp  \sum_{k=0} ^{N-1} \left(   \lambda \norm{\Delta \vec{W}_k }^2  \right) }{2p}\right) \\
&\times \left( \sup_{ N \in \mathbb{N}}  \LnormR{ \sup_{n \in \NZ{N} } \exp  \left(\sup_{u \in \NZ{n} }\sum_{k=u} ^{n-1} \ \beta_k ^N  \right) }{4p}\right)\\
&\times \left( \sup_{ N \in \mathbb{N}}  \LnormR{ (\lambda + \norm{x } )  \sup_{n \in \NZ{N} } \sup_{u \in \NZ{n} }  \prod_{k=u }^{n-1 }\norm{\SG{t_{k+1}},t_{k}   }      }{4p} \right).
\end{split}
\end{align*}
Then by \eqref{eq:exp_ineq} and since $\lambda$ and $x$ are deterministic we get 
\begin{align*}
& \mathop{\sup_{ N \in \mathbb{N}}}_{N \ge 8 \lambda p T }    \LnormR{\sup_{n \in \NZ{N}} D_n ^N}{p} \\
&  \leq e^{\lambda} \left( \lambda + \norm{x } \right)  \left( \mathop{\sup_{ N \in \mathbb{N}}}_{N \ge 8 \lambda p T }   \LnormR{ \exp  \sum_{k=0} ^{N-1} \left(   \lambda \norm{\Delta \vec{W}_k }^2  \right) }{2p}\right) \\
&\times \left( \sup_{ N \in \mathbb{N}}  \LnormR{ \sup_{n \in \NZ{N} } \exp  \left(\sup_{u \in \NZ{n} }\sum_{k=u} ^{n-1} \ \beta_k ^N  \right) }{4p}\right)\\
& \times \left( \sup_{ N \in \mathbb{N}} \prod_{k=0 }^{N-1 } \LnormR{ \exp\left( \norm{ ( A-\frac{1}{2} \sum_{i=1} ^m B_i ^2)\Dt } \right)  \prod_{i=1}^m   \exp \left( 
    \norm{ B_i} \vert \DW{k}  \vert  \right)   }{4p} \right).
\end{align*}

The   second and third     terms are  shown  to be bounded  in  \cite[Lemma 3.5]{hutzenthaler2012}. On the other  hand, 
the boundedness of the last term on the RHS  is due to the $\Dt$ dependence of the upper bound of  $\eval{\exp(\norm{ B_i} \vert \DW{k}  \vert)}$  as shown  in the proof of Lemma \ref{lem:SG_bnd} (ii).
\end{proof}

\begin{lemma} \label{lem:putBound}
Let $\Yt$ be given by \eqref{eq:ctsIntegral}.
For  all $p \geq 2$ and $t \in (t_n,t_{n+1}]$, there exists $K=K(p,T)>0$ such that   
\begin{equation} \label{eq:putBound}
\Lpnorm{\Yt} ^2  \leq K \left(1 +  \Lpnorm{\Yht} ^2\right) .
\end{equation}
\end{lemma}
\begin{proof}
    Note that $\Yt$ is the  solution  to  a linear \Ito SDE 
    $$
    \Yt=\Yht+ \int_{t_n} ^ {t} \left(A \Ys + \SG{s,t_n} \Ft (\Yht) \right) \ds+ \sum_{i=1}^{m}\int_{t_n} ^ {t} \left( B_i  \Yt +\SG{t,t_n } g_i(\Yht ) \right) \dW_t^i
    $$
    with initial condition $\Yht$ on the interval  $[t_n,t]$ by Lemma \ref{lem:itoEq}. After adding and subtracting the terms 
  $\sumi \inttn    \SG{s,t_n}  g_i(0) \dW_{s}^i$, we have by Jensen inequality for sums and integrals as well as the Burkholder-Davis-Gundy inequality (see \cite{maoBook})
     \begin{align*} 
 \eval{\norm{\Yt}}^{p}&   \leq 6 ^{p-1} \left( \eval{\norm{\Yht}^{p}}+\norm{A}^p (t-t_n)^{p-1}  \inttn  \eval{\norm{\Ys}}^{p}  \ds \right.\\
& \left. + (t-t_n)^{p-1}  \inttn \eval{ \norm{\SG{s,t_n} \alpha(\Dt,\Yht)F(\Yht)  }^{p}}\ds \right. \\ 
& \left.+(t-t_n)^{p/2-1} \sum_{i=1}^m \norm{B_i}^p  \inttn  \eval{\norm{  \Ys}^ {p}} \ds \right. \\ 
& \left.+(t-t_n)^{p/2-1} \sum_{i=1}^m  \inttn \eval{ \norm{   \SG{s,t_n}  \left(  g_i(\Yht) -g_i(0)\right) }^ {p}} \ds   \right). 
\end{align*}
Using  Lemma \ref{lem:SG_bnd}, that $ \norm{\alpha(\Dt,\Yhk)F(\Yhk) \Dt} \leq 1$ and the global Lipschitz property of $g_i$, the desired inequality is achieved from an application of the continuous Gronwall Lemma. 
\end{proof}
We now use Lemmas \ref{lem:SG_bnd}--\ref{lem:putBound} to prove the scheme has bounded moments.
\subsection*{Proof  of  Theorem \ref{teo:boundedMoment} : Bounded moments} \label{sec:moments}
\begin{proof}
Consider the \Ito equation for the continuous 
extension   of  numerical  solution  given in Lemma \ref{lem:itoEq} on $[0,t]$  
\begin{multline*}
\Yht=x+\sumk \intk \left(A\Ys+\SG{s,t_k} \alpha(\Dt,\Yhk) F(\Yhk)  \right)\ds \\ +\sumk  \sum_{i=1}^m \intk \left(  B_i\Ys+\SG{s,t_k} g_i(\Yhk)   \right)\dW_{s}^i.
\end{multline*}
By adding and subtracting the terms 
$\sum_{i=1}^m \sumk \intk    \SG{s,t_k}  g_i(0) \dW_{s}^i$ and using the triangle inequality, we have
 \begin{align*} 
 \Lnorm{\Yht}{p} \leq & \norm{x}+\norm{A}  \sumk \intk \Lnorm{\Ys}{p}  \ds \\
& + \sumk \intk \Lnorm{\SG{s,t_k} \alpha(\Dt,\Yhk)F(\Yhk)  }{p}\ds  \\ 
& +\Lnorm{ \sumk \intk \sum_{i=1}^m   B_i\Ys \dW_{s}^i}{p} \\ 
& +\Lnorm{\sumk \intk    \SG{s,t_k}  \sg(0)  \dW_{s}^i}{p} \\ 
& +\Lnorm{\sumk \intk   \SG{s,t_k} \sumi\left( g_i(\Yhk) -g_i(0) \right) \dW_{s}^i  }{p}.
\end{align*}
 By 
Lemma \ref{lem:SG_bnd}, that $ \norm{\alpha(\Dt,\Yhk)F(\Yhk) \Dt} \leq 1 $ in  \eqref{eq:tamingrequirement}, along with the Burkholder-Davis-Gundy inequality  we   have 
  \begin{align*} 
\begin{split}
& \Lnorm{\Yht}{p}  \leq \norm{x}+\norm{A}   \sumk \intk \Lnorm{\Ys}{p}  \ds+ K  N   \\
&+p\left(\sumk \intk \sum_{i=1}^m   \Lnorm{B_i\Ys}{p}^2 \ds\right)^{1/2} \\
& +p K\left(\sum_{i=1}^m \norm{g_i(0)}^2 T \right)^{1/2}+K^2 \sqrt{m} p\left(    \sumk \intk    \Lnorm{\Yhk}{p}^2  \ds\right)^{1/2}.
\end{split}
\end{align*}
By taking the square of both sides and applying  the Jensen inequalities for sums and integrals 
\begin{equation}
\begin{aligned} \label{eq:mixedIneq}
\Lnorm{\Yht}{p} ^2 &  \leq 4 \left( \norm{x} +p K\left(\sum_{i=1}^m \norm{g_i(0)}^2 T \right)^{1/2} + K  N \right)^2 \\ 
& +4 \norm{A}^2 T \sumk \intk   \Lnorm{\Ys}{p} ^2  \ds   \\
& +4 p^2 \sumk \intk  \sum_{i=1}^m   \Lnorm{B_i\Ys}{p}^2 \ds \\
& + 4 K^4 m p^2 \sumk \intk     \Lnorm{\Yhk}{p}^2  \ds.
\end{aligned}
\end{equation}
Substituting 
\eqref{eq:putBound} into  \eqref{eq:mixedIneq} and rewriting the resulting integral inequality in discrete form, we have 
\begin{align*} \label{eq:discrete_Ineq}
\begin{split}
& \Lnorm{Y_n^N}{p} ^2  \leq 4 \left( \norm{x} +p K\left(\sum_{i=1}^m \norm{g_i(0)}^2 T \right)^{1/2}+ K  N \right)^2 \\ 
& \quad + 4 \norm{A}^2 T  \sum_{k=0}^{n-1} K  (1+\Lnorm{Y_k ^N}{p} ^2)  \frac{T}{N}   \\
& \quad +4 p^2   \sum_{k=0}^{n-1}    \sum_{i=1}^m   \norm{B_i}^2 K(1+\Lnorm{Y_k ^N}{p} ^2)  \frac{T}{N}  + 4 K^4 m p^2  \sum_{k=0}^{n-1} \Lnorm{Y_k ^N}{p} ^2  \frac{T}{N}.       
\end{split}
\end{align*}
By the discrete Gronwall inequality
\begin{align*}
\begin{split}
\sup_{ n \in \NZ{N} } \Lnorm{Y_n^N}{p} \leq C_{\norm{A},\norm{B_i}, T,p,K} \left( \norm{x}+p K\left(\sum_{i=1}^m \norm{g_i(0)}^2 T \right)^{1/2} \right .\\ \left . + KN +\norm{A}T\sqrt{K}+p\sum_{i=1}^m \norm{B_i}\sqrt{K T} \right).
\end{split}
\end{align*}
Following the bootstrap argument  presented in \cite[Lemma 3.9]{hutzenthaler2012} to deal with the term $KN$ on the RHS, we get 
\begin{equation}\label{eq:complement_ineq}
\sup_{N \in  \mathbb{N} }\sup_{n \in \NZ{N}  } \Lnorm{1_{(\Omega_n^N)^c}Y_n^N }{p}< \infty 
\end{equation}
and 
\begin{equation}\label{eq:D_ineq}
\mathop{\sup_{N \in  \mathbb{N}}}_{N \ge 8\lambda p T  }\sup_{n \in \NZ{N}   } \Lnorm{1_{\Omega_n^N} Y_n^N}{p}\leq 
\mathop{\sup_{N \in  \mathbb{N}}}_{N \ge 8\lambda p T  }\sup_{n \in \NZ{N}   }  \Lnorm{D_n^N}{p} .
\end{equation}
The boundedness of the term in right hand side of \eqref{eq:D_ineq} is proved in Lemma \ref{lem:Dboundedness}. Hence the proof is complete.
\end{proof}

\section{Proof of Weak Convergence} 
\label{sec:WeakConvergence}

We now take $C$ as a generic constant that is independent of the time step size $\Dt$.
\begin{lemma}
\label{lem:eq:growth_ineq_stoch_1}
Let the Assumptions of Theorem \ref{teo:KolmogPDE} hold.
For each $T>0$, there exists $C>0$ such that for $0 \leq t \leq T$  
\begin{equation} \label{eq:growth_ineq_det_1}
\vert \eval{\vec{D}  \Psi(t,x)\xi }\vert \leq  C \norm{\eval{\xi}} , \qquad \text{for all } \xi \in L^p(\Omega, \mathbb{R}^d), \quad p\geq1.
\end{equation}
\end{lemma}
 \begin{proof}
By interchanging the derivative and expectation (see discussion in \cite[sections 5.1 and 5.2]{bossy2021weak}), the chain rule and the definition of $\Psi$ in \eqref{eq:psidef}
 \begin{align*} \label{eq:BossyInterchange}
\vert \eval{\vec{D}  \Psi(t,x)\xi }\vert&=\vert \vec{D} \Psi(t,x) \eval{\xi}  \vert \\ 
& =\vert \eval{\vec{D} \phi(X_t ^x) \vec{D}_x X_t ^x } \eval{\xi}  \vert.
\end{align*}
Then by the deterministic Cauchy-Schwarz and triangle inequality 
\begin{align*} 
\vert \eval{\vec{D}  \Psi(t,x)\xi }\vert
&\leq \norm{\eval{\vec{D} \phi(X_t ^x) \vec{D}_x X_t ^x }} \norm{\eval{\xi} } \nonumber \\
& \leq \eval{\norm{\vec{D} \phi(X_t ^x) \vec{D}_x X_t ^x }} \norm{\eval{\xi} }\nonumber \\
& \leq  \eval{\norm{\vec{D} \phi(X_t ^x) }^2} ^{1/2} \eval{\norm{ \vec{D}_x X_t ^x }^2} ^{1/2}  \norm{\eval{\xi} }
\end{align*}
where on the last step we used the stochastic Cauchy-Schwarz  inequality.
The fact that $\phi\in C^2_b(\mathbb{R}^d)$ and Theorem \ref{teo:bounded_derivatives} complete  the proof.
\end{proof}
Considering the definition of the function $\Psi$  in \eqref{eq:psidef}, we define
\begin{equation} \label{eq:udef}
 u(t,\bar{Y}_t) :=\Psi (T-t, \bar{Y}_t).
\end{equation}

We point out the inequality \eqref{eq:growth_ineq_det_1}  implies, for $u$ defined in \eqref{eq:udef}, that 
\begin{equation}\label{eq:growth_ineq_stoch_1_U}
\vert \eval{\vec{D} u(s,\Ys)\xi} \vert \leq C \norm{ \eval{\xi} }
\end{equation}
where  $\xi$  is a  $\mathbb{R}^d$-valued random variable. 
%
We are now in a position to prove weak convergence.
\subsection*{Proof of Theorem \ref{thrm:1}: weak convergence} 
\begin{proof}
Recall the definition of $\Psi(t,x)$ from \eqref{eq:psidef}.
Applying the \Ito formula on $[t_n,t_{n+1})$ for $\Psi(t,\Yt)$ with $\Yt$ given by \eqref{eq:cts}  gives
\begin{equation} \label{eq:Ito_numeric}
\d\Psi(t, \Yt)=\left( \der{}{t}+\Lh_n(t) \right) \Psi(t,\Yt)\dt+\sum_{i=1}^ m \Lhi_n (t) \Psi(t,\Yt) \d W^i_t 
\end{equation}
where 
\begin{align*}
    \begin{split}
\Lh_n (t) \Psi(t,\Yt)  & :=\vec{D} \Psi(t,\Yt) \mutp{t,t_n}+\frac{1}{2} \sum_{i=1}^ m  \sigmatp{t,t_n} ^ \intercal \vec{D}^2  \Psi(t,\Yt)  \sigmatp{t,t_n} \\
\Lhi_n (t)\Psi(t,\Yt) & :=\vec{D} \Psi(t,\Yt) \sigmatp{t,t_n}
\end{split}
\end{align*}
and $\mut$ and $\sigmat$  are defined in \eqref{eq:mutm_sigmatm_def}.
Now we define the error function $e$ by 
$$
e:=\vert \Psi(0, \bar{Y}_T)-\Psi(T ,\bar{Y}_0)\vert =\bigg \vert \eval{ \phi(\bar{Y}_T) -\phi(X_T^x) \vert X_0=x} \bigg \vert.
$$
By definition of $u$ in \eqref{eq:udef} and by a telescopic sum, we have
$$
e=\vert \eval{ u(T,\bar{Y}_T)}-u(0,\bar{Y}_0)\vert =\Big \vert  \sum_{k=0}^{N-1} \eval{ u(t_{k+1},\bar{Y}_{t_{k+1})}-u(t_{k},\bar{Y}_{t_k}) }\Big\vert.
$$
An application of \Ito's formula for $u(s,\Ys)$,  
using \eqref{eq:Ito_numeric} on each of the subintervals, and the zero expectation of \Ito integrals gives
$$
e=\Bigg\vert \sum_{k=0}^{N-1} \eval {\int _{s=t_k}^{s=t_{k+1}} \left( \der{ } {s}   u(s, \bar{Y}_s) +  \Lh_k(s) u(s,\Ys)  \right) \ds } \Bigg\vert.
$$

On the other hand, $u(s,\Ys)=\Psi(T-s, \Ys) $  satisfies the Kolmogorov PDE \eqref{eq:KolmogorovPDE}
$$\der{}{s} u(s,\Ys)=-\LL u(s,\Ys).$$
So we have
\begin{equation}
\label{eq:weakerr1}
e=\bigg\vert \sum_{k=0}^{N-1} \eval{\int_{t_k}^{t_k+1} \left( \Lh_k(s) u(s,\Ys)-\LL u(s,\Ys) \right) \ds }\bigg\vert.
\end{equation}
Consider the integrand  
\begin{align*} \label{eq:term_inside_integral }
\begin{split}
  \Lh_k(s) &  u(s,\Ys) -\LL u(s,\Ys)  = \vec{D} u(s,\Ys) \left( \mutp{s,t_k}-\mu(\Ys)\right)  \\
  & +\frac{1}{2} \sum_{i=1}^ m \sigmatp{s,t_k} ^\intercal \vec{D}^2  u(s,\Ys)  \sigmatp{s,t_k}  
  - \frac{1}{2} \sum_{i=1}^ m  \sigma_i(\Ys)^\intercal  \vec{D}^2  u(s,\Ys)  \sigma_i(\Ys)  \\
  &=  \vec{ D} u(s,\Ys) \left( \SG{s,t_k } \Ft (\Yhk) - F(\Ys) \right) \\ &  +\sum_{i=1}^ m  (B_i \Ys)^\intercal \vec{D}^2  u(s,\Ys) \left( \SG{s,t_k } g_i (\Yhk) - g_i(\Ys)\right)     \\
  &+\frac{1}{2}\sum_{i=1}^ m \left( \SG{s,t_k } g_i (\Yhk)\right)^\intercal  \vec{D} ^2  u(s,\Ys) \left( \SG{s,t_k } g_i (\Yhk)\right) \\ 
  &   -\frac{1}{2} \sum_{i=1}^ m  g_i (\Ys) ^\intercal \vec{D} ^2  u(s,\Ys)   g_i (\Ys).
 \end{split}
\end{align*}
Adding  and subtracting the terms $\vec{ D} u(s,\Ys) \Ft(\Yhk)  $, $\vec{D} u(s,\Ys)F(\Yhk)$ and defining 
\begin{align*}
\Theta_1^k(s,\vec{x})&:=  \sum_{i=1}^ m \left( B_i \Ys \right) ^ \intercal \vec{D}^2  u(s,\vec{x}) \left( \SG{s,t_k } g_i (\Yhk) \right)   \\
\Theta_2^k(s,\vec{x})&:=\sum_{i=1}^ m \left( B_i \Ys \right) ^ \intercal \vec{D}^2  u(s,\vec{x})   g_i(\vec{x}) \\
\Theta_3^k(s,\vec{x})&:=  \frac{1}{2} \sum_{i=1}^ m  \left(\SG{s,t_k}g_i(\Yhk)\right)^\intercal \vec{D}^2  u(s,\vec{x})  \left(\SG{s,t_k}g_i(\Yhk)\right) \\
\Theta_4^k(s,\vec{x})&:=\frac{1}{2} \sum_{i=1}^ m  g_i(\vec{x}) ^\intercal  \vec{D}^2  u(s,\vec{x}) g_i(\vec{x}) 
\end{align*}
such that  
  $\Theta_1^k(t_k,  \Yhk)=\Theta_2^k(t_k,  \Yhk)$, $\Theta_3^k(t_k,  \Yhk)=\Theta_4^k(t_k,  \Yhk)$, we have
$$
 \Lh_k(s) u(s,\Ys)-\LL u(s,\Ys):=T_1^k(s)+T_2^k(s)+T_3^k(s)+T_4^k(s)+T_5^k(s)+T_6^k(s)+T_7^k(s)
$$
where 
\begin{align*}
T_1^k(s) &:=\vec{D} u(s,\Ys) \left(\SG{s,t_k} \Ft(\Yhk)- \Ft(\Yhk)\right) & T_4 ^k (s) & :=  \Theta_1^k(s,  \Ys)-\Theta_1^k(t_k,  \Yhk)\\
T_2 ^k(s) &:=   \vec{D} u(s,\Ys)  \left( \Ft(\Yhk)- F(\Yhk) \right)   & T_5^k (s)& :=  \Theta_2^k(t_k,  \Yhk)-\Theta_2^k(s,  \Ys)\\
T_3 ^k(s) & :=  \vec{D} u(s,\Ys) \left( F(\Yhk)-F(\Ys) \right) 
& T_6 ^k (s)& :=  \Theta_3^k(s,  \Ys)-\Theta_3^k(t_k,  \Yhk) \\ 
& & T_7 ^k (s)& := \Theta_4^k(t_k,  \Yhk)- \Theta_4^k(s,  \Ys).
\end{align*}
In terms of these functions, the error \eqref{eq:weakerr1} is given by
$$
e=\left\vert \sum_{k=0}^{N-1} \eval{\intk \left(T_1^k(s)+T_2^k(s)+T_3^k(s)+T_4^k(s)+T_5^k(s)+T_6^k(s)+T_7^k(s) \right)\ds} \right\vert.
$$
We now seek to bound each $\left\vert \eval{T_i^k}\right\vert$ by a term of order $(s-t_k)$ or $\Dt$ on the subinterval $[t_k,t_{k+1}]$, for $i=1,\ldots,7$.

Starting with $T_1^k$, the \Ito  equation for $\SG{}$  in \eqref{eq:Homogen} for $\mathcal{F} _ {t_k}$  measurable, $\mathbb{R}^d$-valued random variables $v\in L^2  (\Omega,\mathbb{R}^d)$     gives 
\begin{equation*}
\SG{s,t_k} v=   v+ \int_{t_k} ^s A \SG{r,t_k} vdr+ \int_{t_k} ^s \sB  \SG{r,t_k} vdW^i _r.
\end{equation*}
By the zero expectation of \Ito integrals and Fubini's Theorem,
\begin{equation*} \label{eq:expectation_integral}
\eval{\SG{s,t_k} v-v} =    \int_{t_k} ^s \eval{\SG{r,t_k} v}dr.
\end{equation*}
Taking $v=\Ft(\Yhk)$,  boundedness of the operator  $\SG{}$ from Lemma \ref{lem:SG_bnd}  and  bounded 
moments of numerical solution in Theorem  \ref{teo:boundedMoment} and the inequality  \eqref{eq:growth_ineq_stoch_1_U}, we have 
\begin{equation*}
\left\vert \eval{T_1^k(s)} \right\vert \leq C \norm {\int_{t_k} ^s \eval{\SG{r,t_k} \Ft(\Yhk)}dr}  \leq    C (s-t_k).
\end{equation*}
For $T_2^k$, by the definition of $\Ft$ in \eqref{eq:tamingrequirement} we have
$$
\left\vert \eval{T_2^k(s)} \right\vert =\left\vert \eval{  \vec{D} u(s,\Ys) \left(-\Dt    \alpha(\Dt,\Yhk)\norm{F(\Yhk)}    F(\Yhk)  \right) } \right\vert  \leq C \Dt ,
$$ 
where the inequality $ \alpha(\Dt,\Yhk)  < 1$ from \eqref{eq:tamingrequirement}, bounded moments of numerical solution from Theorem \ref{teo:boundedMoment} and the inequality  \eqref{eq:growth_ineq_stoch_1_U}  are  used. 

We now consider $T_3^k$. \Ito's formula for  $F_{j}(\Ys)$ around $\Yhk$ where $F_{j}:\mathbb{R}^d \to \mathbb{R}$ refers to the $j$th component of the vector valued function $F$ for $j=1,\hdots,d$  gives 
\begin{equation*}
F_{j}(\Ys)=F_{j}(\Yhk)+\int_{t_k} ^s \Lh_k   F_{j}(\Yr)  dr 
+\sum_{i=1}^ m  \int_{t_k} ^s \Lhi_k   F_{j}(\Yr) dW^i _r
\end{equation*}
where 
\begin{align*}
    \begin{split}
\Lh_k   F_{j}(\Yr)  & :=\vec{D} F_{j}(\Yr) \mutp{r,t_k}+\frac{1}{2} \sum_{i=1}^ m  \sigmatp{r,t_k} ^ \intercal \vec{D}^2  F_{j}(\Yr)  \sigmatp{r,t_k} \\
\Lhi_k   F_{j}(\Yr) & :=\vec{D} F_{j}(\Yr)    \sigmatp{r,t_k}. 
\end{split}
\end{align*}
By  zero expectation of the \Ito integral and Jensen's inequality for the integral and expectation, we have
\begin{equation*}
\left\vert \eval{F_{j}(\Ys)-F_{j}(\Yhk)} \right\vert ^2=(s-t_k)\int_{t_k} ^s \eval{\vert \Lh_k   F_{j}(\Yr)   \vert ^2} dr .
\end{equation*}
By polynomial growth assumption on the derivatives of $F$ and  bounded numerical moments, we conclude that 
$$
\left\vert \eval{T_3^k(s)} \right\vert =C\norm{\eval{F(\Ys)-F(\Yhk)}}\leq C (s-t_k).
$$
Now consider the term $T_4^k$. 
By applying the \Ito formula  to the function $\Theta_1^k(t,\Yt)$ over the interval $[t_k,s]$ and taking expectation, we have
\begin{align*}
T_4^k  = \eval{\Theta_1^k(s,  \Ys)-\Theta_1^k(t_k,  \Yhk)}=\eval{  \int_{t_k} ^s \left((\der{}{r}+ \Lh_k)\Theta_1^k(r,\Yr) \right) dr} .  
\end{align*}
The integrand is given by
\begin{align}\label{eq:Theta_Ito}
   \bigg(\der{}{r}+& \Lh_k\bigg)\Theta_1^k(r,\Yr) = \sum_{i=1}^ m \left(  B_i \Yr \right) ^ \intercal \der{}{r} \left(  \vec{D}^2  u(r,\Yr) \right) \left( \SG{r,t_k } g_i (\Yhk) \right)  \nonumber \\ 
   &+\sum_{i=1}^ m \left(  B_i \Yr \right) ^ \intercal   \vec{D}^2  u(r,\Yr) ( A-\frac{1}{2} \sum_{k=1} ^m B_k ^2) \SG{r,t_k } g_i (\Yhk)  \nonumber \\ 
   &+\sum_{i=1}^ m \vec{D} \Big[ \left( B_i \Yr \right) ^ \intercal \vec{D}^2  u(r,\Yr) \left( \SG{r,t_k } g_i (\Yhk) \right) \Big]  \mutp{r,t_k} \nonumber \\ 
&+\frac{1}{2} \sum_{i=1}^ m  \sigmatp{r} ^ \intercal \vec{D}^2  \Big[ \left( B_i \Yr \right) ^ \intercal \vec{D}^2  u(r,\Yr) \left( \SG{r,t_k } g_i (\Yhk) \right) \Big]    \sigmatp{r,t_k}.
\end{align}
By the continuity of all the derivatives of $u$, the condition that $\phi\in C^4_b(\mathbb{R}^d)$ and bounded moments of the numerical method in Theorem \ref{teo:boundedMoment} we have, similar to the arguments in \cite[Appendix B]{wang2021weak}, that $\eval{   (\der{}{r}+ \Lh_k)\Theta_1^k(r,\Yr)  }$ is uniformly bounded on $[t_k,s]$. Therefore we conclude  
\begin{align*}
\left\vert \eval{T_4^k (s)} \right\vert  
\leq   \int_{t_k} ^s \vert\eval{(\der{}{r}+ \Lh_k)\Theta_1^k(r,\Yr)} \vert dr  \leq C (s-t_k).
\end{align*}

Similar arguments can be applied to $T_5^k,T_6^k,T_7^k$ as well. Finally, by Fubini's theorem for integrals
\begin{align*}
  e&=\left\vert \sum_{k=0}^{N-1} \eval{\intk (T_1^k(s)+T_2^k(s)+T_3^k(s)+T_4^k(s)+T_5^k(s)+T_6^k(s)+T_7^k(s))\ds} \right\vert  \\
  & \leq C  \left(\Dt+  \sum_{k=0}^{N-1} \intk (s-t_k)\ds\right) \\
  & = C(T) \Dt.
\end{align*}
we obtain the desired order for the weak error of the scheme. 
\end{proof}

\begin{acknowledgements}
The first author is supported  by  The Scientific and Technological Research Council of Turkey  (TUBITAK) 2219-International Postdoctoral Research Fellowship Program, 
grant number 1059B192202051. We are grateful to two anonymous reviewers for their comments and suggestions.
\end{acknowledgements}

\noindent
{\small \textbf{Conflict of interests:}\\
Not Applicable.}

\bibliographystyle{spmpsci}      
\bibliography{taming_references}   

\begin{thebibliography}{10}
\providecommand{\url}[1]{{#1}}
\providecommand{\urlprefix}{URL }
\expandafter\ifx\csname urlstyle\endcsname\relax
  \providecommand{\doi}[1]{DOI~\discretionary{}{}{}#1}\else
  \providecommand{\doi}{DOI~\discretionary{}{}{}\begingroup
  \urlstyle{rm}\Url}\fi

\bibitem{abladinger2017}
Ableidinger, M., Buckwar, E., Thalhammer, A.: An importance sampling technique
  in {Monte Carlo} methods for {SDEs} with a.s. stable and mean-square unstable
  equilibrium.
\newblock Journal of Computational and Applied Mathematics \textbf{316}, 3--14
  (2017).
\newblock \doi{https://doi.org/10.1016/j.cam.2016.08.043}.
\newblock
  \urlprefix\url{https://www.sciencedirect.com/science/article/pii/S0377042716304125}.
\newblock Selected Papers from NUMDIFF-14

\bibitem{Beyn2016}
Beyn, W.J., Isaak, E., Kruse, R.: Stochastic {C-Stability} and {B-Consistency}
  of explicit and implicit {Euler}-type schemes.
\newblock Journal of Scientific Computing \textbf{67}(3), 955--987 (2016).
\newblock \doi{10.1007/s10915-015-0114-4}

\bibitem{biscay1996}
Biscay, R., Jimenez, J., Riera, J., Valdes, P.: Local linearization method for
  the numerical solution of stochastic differential equations.
\newblock Annals of the Institute of Statistical Mathematics \textbf{48}(4),
  631--644 (1996)

\bibitem{bossy2021weak}
Bossy, M., Jabir, J.F., Martinez, K.: On the weak convergence rate of an
  exponential {Euler} scheme for {SDEs} governed by coefficients with
  superlinear growth.
\newblock Bernoulli \textbf{27}(1), 312--347 (2021)

\bibitem{brehier2020weakergodic}
{Br\'ehier}, C.E.: Approximation of the invariant distribution for a class of
  ergodic sdes with one-sided lipschitz continuous drift coefficient using an
  explicit tamed euler scheme.
\newblock ESAIM: PS \textbf{27}, 841--866 (2023).
\newblock \doi{10.1051/ps/2023017}.
\newblock \urlprefix\url{https://doi.org/10.1051/ps/2023017}

\bibitem{brehier2020weak}
Br{\'e}hier, C.E., Gouden{\`e}ge, L.: Weak convergence rates of splitting
  schemes for the stochastic {Allen}--{Cahn} equation.
\newblock BIT Numerical Mathematics \textbf{60}(3), 543--582 (2020)

\bibitem{cai2021weak}
Cai, M., Gan, S., Wang, X.: Weak convergence rates for an explicit
  full-discretization of stochastic {Allen}--{Cahn} equation with additive
  noise.
\newblock Journal of Scientific Computing \textbf{86}(3), 1--30 (2021)

\bibitem{cerrai}
Cerrai, S.: Second order PDE’s in finite and infinite dimension: a
  probabilistic approach, 1 edn.
\newblock Springer, Berlin Heidelberg (2001).
\newblock \doi{https://doi.org/10.1007/b80743}

\bibitem{CHEN2020135}
Chen, Z., Gan, S., Wang, X.: A full-discrete exponential {Euler} approximation
  of the invariant measure for parabolic stochastic partial differential
  equations.
\newblock Applied Numerical Mathematics \textbf{157}, 135--158 (2020).
\newblock \doi{https://doi.org/10.1016/j.apnum.2020.05.008}.
\newblock
  \urlprefix\url{https://www.sciencedirect.com/science/article/pii/S0168927420301525}

\bibitem{debrabant2021rk}
Debrabant, K., Kv{\ae}rn{\o}, A., Mattsson, N.C.: {Runge}--{Kutta} {Lawson}
  schemes for stochastic differential equations.
\newblock BIT Numerical Mathematics \textbf{61}(2), 381--409 (2021)

\bibitem{utkuLord}
Erdogan, U., Lord, G.J.: A new class of exponential integrators for {SDE}s with
  multiplicative noise.
\newblock IMA Journal of Numerical Analysis \textbf{39}(2), 820--846 (2018).
\newblock \doi{10.1093/imanum/dry008}

\bibitem{tamedGBM}
Erdogan, U., Lord, G.J.: Strong convergence of a {GBM} based tamed integrator
  for {SDEs} and an adaptive implementation.
\newblock Journal of Computational and Applied Mathematics \textbf{399},
  113,704 (2022).
\newblock \doi{https://doi.org/10.1016/j.cam.2021.113704}.
\newblock
  \urlprefix\url{https://www.sciencedirect.com/science/article/pii/S0377042721003265}

\bibitem{giles_2015}
Giles, M.B.: Multilevel {Monte} {Carlo} methods.
\newblock Acta Numerica \textbf{24}, 259–328 (2015).
\newblock \doi{10.1017/S096249291500001X}

\bibitem{hutzenthaler2011}
Hutzenthaler, M., Jentzen, A., Kloeden, P.E.: Strong and weak divergence in
  finite time of {Euler}'s method for stochastic differential equations with
  non-globally {Lipschitz} continuous coefficients.
\newblock Proceedings of the Royal Society A \textbf{467}(2130), 1563--1576
  (2011).
\newblock \doi{10.1098/rspa.2010.0348}

\bibitem{hutzenthaler2012}
Hutzenthaler, M., Jentzen, A., Kloeden, P.E.: Strong convergence of an explicit
  numerical method for {SDE}s with nonglobally {Lipschitz} continuous
  coefficients.
\newblock Ann. Appl. Probab. \textbf{22}(4), 1611--1641 (2012).
\newblock \doi{10.1214/11-AAP803}

\bibitem{expmil}
Jentzen, A., R{\"o}ckner, M.: A {M}ilstein scheme for {SPDEs}.
\newblock Foundations of Computational Mathematics \textbf{15}(2), 313--362
  (2015)

\bibitem{jimenez1999simulation}
Jimenez, J., Shoji, I., Ozaki, T.: Simulation of stochastic differential
  equations through the local linearization method. a comparative study.
\newblock Journal of Statistical Physics \textbf{94}(3-4), 587--602 (1999)

\bibitem{JimenezCarbonell}
Jimenez, J.C., Carbonell, F.: Convergence rate of weak local linearization
  schemes for stochastic differential equations with additive noise.
\newblock J. Comput. Appl. Math. \textbf{279}, 106--122 (2015).
\newblock \doi{10.1016/j.cam.2014.10.021}.
\newblock \urlprefix\url{http://dx.doi.org/10.1016/j.cam.2014.10.021}

\bibitem{kloeden2011}
Kloeden, P., Platen, E.: Numerical Solution of Stochastic Differential
  Equations.
\newblock Stochastic Modelling and Applied Probability. Springer, Berlin
  Heidelberg (2011)

\bibitem{lang2018}
Lang, A., Petersson, A.: {Monte} {Carlo} versus multilevel {Monte Carlo} in
  weak error simulations of {SPDE} approximations.
\newblock Mathematics and Computers in Simulation \textbf{143}, 99--113 (2018).
\newblock \doi{https://doi.org/10.1016/j.matcom.2017.05.002}.
\newblock
  \urlprefix\url{https://www.sciencedirect.com/science/article/pii/S037847541730188X}.
\newblock 10th IMACS Seminar on Monte Carlo Methods

\bibitem{TruncatedMao2018}
Li, X., Mao, X., Yin, G.: {Explicit numerical approximations for stochastic
  differential equations in finite and infinite horizons: truncation methods,
  convergence in pth moment and stability}.
\newblock IMA Journal of Numerical Analysis \textbf{39}(2), 847--892 (2018).
\newblock \doi{10.1093/imanum/dry015}

\bibitem{Lord20141}
Lord, G.J., Powell, C.E., Shardlow, T.: An introduction to computational
  stochastic PDEs (2014).
\newblock \doi{10.1017/CBO9781139017329}

\bibitem{lord2004}
Lord, G.J., Rougemont, J.: A numerical scheme for stochastic {PDEs} with
  {Gevrey} regularity.
\newblock IMA J. of Numerical Analysis \textbf{24}(4), 587--604 (2004)

\bibitem{maoBook}
Mao, X.: Stochastic differential equations and applications.
\newblock Elsevier, Amsterdam (2007)

\bibitem{TruncatedMao2015}
Mao, X.: The truncated {Euler}--{Maruyama} method for stochastic differential
  equations.
\newblock Journal of Computational and Applied Mathematics \textbf{290}, 370 --
  384 (2015).
\newblock \doi{https://doi.org/10.1016/j.cam.2015.06.002}

\bibitem{Mora}
Mora, C.M.: Weak exponential schemes for stochastic differential equations with
  additive noise.
\newblock IMA J. Numer. Anal. \textbf{25}(3), 486--506 (2005).
\newblock \doi{10.1093/imanum/dri001}.
\newblock \urlprefix\url{http://dx.doi.org/10.1093/imanum/dri001}

\bibitem{sabanis2013note}
Sabanis, S.: A note on tamed {Euler} approximations.
\newblock Electronic Communications in Probability \textbf{18}, 1--10 (2013)

\bibitem{sabanis2016}
Sabanis, S.: Euler approximations with varying coefficients: The case of
  superlinearly growing diffusion coefficients.
\newblock Ann. Appl. Probab. \textbf{26}(4), 2083--2105 (2016).
\newblock \doi{10.1214/15-AAP1140}

\bibitem{TretyakovZhang}
Tretyakov, M.V., Zhang, Z.: A fundamental mean-square convergence theorem for
  {SDEs} with locally {Lipschitz} coefficients and its applications.
\newblock SIAM J. Numer. Anal. \textbf{51}(6), 3135--3162 (2013).
\newblock \doi{10.1137/120902318}

\bibitem{wang2021weakIMA}
Wang, X., Zhao, Y., Zhang, Z.: {Weak error analysis for strong approximation
  schemes of SDEs with super-linear coefficients}.
\newblock IMA Journal of Numerical Analysis p. drad083 (2023).
\newblock \doi{10.1093/imanum/drad083}.
\newblock \urlprefix\url{https://doi.org/10.1093/imanum/drad083}

\bibitem{yang2021class}
Yang, G., Burrage, K., Komori, Y., Burrage, P., Ding, X.: A class of new
  {Magnus}-type methods for semi-linear non-commutative {Ito} stochastic
  differential equations.
\newblock Numerical Algorithms pp. 1--25 (2021)

\bibitem{wang2021weak}
Zhao, Y., Wang, X.: Weak approximation schemes for {SDE}s with super-linearly
  growing coefficients.
\newblock Applied Numerical Mathematics \textbf{198} (2024).
\newblock \doi{10.1016/j.apnum.2024.01.003}

\bibitem{Izgi2018}
İzgi, B., Çetin, C.: Semi-implicit split-step numerical methods for a class
  of nonlinear stochastic differential equations with non-{Lipschitz} drift
  terms.
\newblock Journal of Computational and Applied Mathematics \textbf{343}, 62 --
  79 (2018).
\newblock \doi{https://doi.org/10.1016/j.cam.2018.03.027}

\end{thebibliography}

\end{document}